\documentclass[11pt, a4paper, twoside]{amsart}
\usepackage[centering, totalwidth = 370pt, totalheight = 610pt]{geometry}
\usepackage{amssymb, amsmath, amsthm, enumerate, microtype, stmaryrd, url}
\usepackage[latin1]{inputenc}
\usepackage[dvips, arrow, matrix, tips, curve]{xy}
\usepackage{color}
\definecolor{darkgreen}{rgb}{0,0.45,0}
\usepackage[pagebackref,colorlinks,citecolor=darkgreen,linkcolor=darkgreen]{hyperref}
\SelectTips{cm}{10}

\newcommand{\cat}[1]{\mathbf{#1}}
\newcommand{\tcat}[1]{\mathbf{#1}}
\newcommand{\op}{\mathrm{op}}
\newcommand{\id}{\mathrm{id}}
\newcommand{\thg}{{\mathord{\text{--}}}}
\newcommand{\wo}{\mathrel{\pitchfork}}

\newcommand{\abs}[1]{{\left|{#1}\right|}}

\newcommand{\set}[2]{\left\{\,#1 \ \vrule\  #2\,\right\}}
\newcommand{\spn}[1]{{\left<{#1}\right>}}

\newcommand{\defeq}{\mathrel{\mathop:}=}

\newcommand{\cd}[2][]{\vcenter{\hbox{\xymatrix#1{#2}}}}

\newcommand{\A}{{\mathcal A}}
\newcommand{\B}{{\mathcal B}}
\newcommand{\C}{{\mathcal C}}

\renewcommand{\L}{{\mathcal L}}
\newcommand{\M}{{\mathcal M}}

\renewcommand{\O}{{\mathcal O}}

\newcommand{\R}{{\mathcal R}}
\newcommand{\T}{{\mathcal T}}
\newcommand{\U}{{\mathcal U}}
\newcommand{\V}{{\mathcal V}}

\newcommand{\X}{{\mathcal X}}

\newcommand{\xtor}[1]{\cdl[@1]{{} \ar[r]|-{\object@{|}}^{#1} & {}}}

\makeatletter

\def\hookleftarrowfill@{\arrowfill@\leftarrow\relbar{\relbar\joinrel\rhook}}
\def\twoheadleftarrowfill@{\arrowfill@\twoheadleftarrow\relbar\relbar}
\def\leftbararrowfill@{\arrowdoublefill@{\leftarrow\mkern-5mu}\relbar\mapstochar\relbar\relbar}
\def\Leftbararrowfill@{\arrowdoublefill@{\Leftarrow\mkern-2mu}\Relbar\Mapstochar\Relbar\Relbar}
\def\leftringarrowfill@{\arrowdoublefill@{\leftarrow\mkern-3mu}\relbar{\mkern-3mu\circ\mkern-2mu}\relbar\relbar}
\def\lefttriarrowfill@{\arrowfill@{\mathrel\triangleleft\mkern0.5mu\joinrel\relbar}\relbar\relbar}
\def\Lefttriarrowfill@{\arrowfill@{\mathrel\triangleleft\mkern1mu\joinrel\Relbar}\Relbar\Relbar}

\def\hookrightarrowfill@{\arrowfill@{\lhook\joinrel\relbar}\relbar\rightarrow}
\def\twoheadrightarrowfill@{\arrowfill@\relbar\relbar\twoheadrightarrow}
\def\rightbararrowfill@{\arrowdoublefill@{\relbar\mkern-0.5mu}\relbar\mapstochar\relbar\rightarrow}
\def\Rightbararrowfill@{\arrowdoublefill@{\Relbar\mkern-2mu}\Relbar\Mapstochar\Relbar\Rightarrow}
\def\rightringarrowfill@{\arrowdoublefill@\relbar\relbar{\mkern-2mu\circ\mkern-3mu}\relbar{\mkern-3mu\rightarrow}}
\def\righttriarrowfill@{\arrowfill@\relbar\relbar{\relbar\joinrel\mkern0.5mu\mathrel\triangleright}}
\def\Righttriarrowfill@{\arrowfill@\Relbar\Relbar{\Relbar\joinrel\mkern1mu\mathrel\triangleright}}

\def\leftrightarrowfill@{\arrowfill@\leftarrow\relbar\rightarrow}
\def\mapstofill@{\arrowfill@{\mapstochar\relbar}\relbar\rightarrow}

\renewcommand*\xleftarrow[2][]{\ext@arrow 20{20}0\leftarrowfill@{#1}{#2}}
\providecommand*\xLeftarrow[2][]{\ext@arrow 60{22}0{\Leftarrowfill@}{#1}{#2}}
\providecommand*\xhookleftarrow[2][]{\ext@arrow 10{20}0\hookleftarrowfill@{#1}{#2}}
\providecommand*\xtwoheadleftarrow[2][]{\ext@arrow 60{20}0\twoheadleftarrowfill@{#1}{#2}}
\providecommand*\xleftbararrow[2][]{\ext@arrow 10{22}0\leftbararrowfill@{#1}{#2}}
\providecommand*\xLeftbararrow[2][]{\ext@arrow 50{24}0\Leftbararrowfill@{#1}{#2}}
\providecommand*\xleftringarrow[2][]{\ext@arrow 10{26}0\leftringarrowfill@{#1}{#2}}
\providecommand*\xlefttriarrow[2][]{\ext@arrow 80{24}0\lefttriarrowfill@{#1}{#2}}
\providecommand*\xLefttriarrow[2][]{\ext@arrow 80{24}0\Lefttriarrowfill@{#1}{#2}}

\renewcommand*\xrightarrow[2][]{\ext@arrow 01{20}0\rightarrowfill@{#1}{#2}}
\providecommand*\xRightarrow[2][]{\ext@arrow 04{22}0{\Rightarrowfill@}{#1}{#2}}
\providecommand*\xhookrightarrow[2][]{\ext@arrow 00{20}0\hookrightarrowfill@{#1}{#2}}
\providecommand*\xtwoheadrightarrow[2][]{\ext@arrow 03{20}0\twoheadrightarrowfill@{#1}{#2}}
\providecommand*\xrightbararrow[2][]{\ext@arrow 01{22}0\rightbararrowfill@{#1}{#2}}
\providecommand*\xRightbararrow[2][]{\ext@arrow 04{24}0\Rightbararrowfill@{#1}{#2}}
\providecommand*\xrightringarrow[2][]{\ext@arrow 01{26}0\rightringarrowfill@{#1}{#2}}
\providecommand*\xrighttriarrow[2][]{\ext@arrow 07{24}0\righttriarrowfill@{#1}{#2}}
\providecommand*\xRighttriarrow[2][]{\ext@arrow 07{24}0\Righttriarrowfill@{#1}{#2}}

\providecommand*\xmapsto[2][]{\ext@arrow 01{20}0\mapstofill@{#1}{#2}}
\providecommand*\xleftrightarrow[2][]{\ext@arrow 10{22}0\leftrightarrowfill@{#1}{#2}}
\providecommand*\xLeftrightarrow[2][]{\ext@arrow 10{27}0{\Leftrightarrowfill@}{#1}{#2}}

\makeatother

\newcommand{\twocong}[2][0.5]{\ar@{}[#2] \save ?(#1)*{\cong}\restore}
\newcommand{\twoeq}[2][0.5]{\ar@{}[#2] \save ?(#1)*{=}\restore}
\newcommand{\rtwocell}[3][0.5]{\ar@{}[#2] \ar@{=>}?(#1)+/l 0.2cm/;?(#1)+/r 0.2cm/^{#3}}
\newcommand{\ltwocell}[3][0.5]{\ar@{}[#2] \ar@{=>}?(#1)+/r 0.2cm/;?(#1)+/l 0.2cm/^{#3}}
\newcommand{\ltwocello}[3][0.5]{\ar@{}[#2] \ar@{=>}?(#1)+/r 0.2cm/;?(#1)+/l 0.2cm/_{#3}}
\newcommand{\dtwocell}[3][0.5]{\ar@{}[#2] \ar@{=>}?(#1)+/u  0.2cm/;?(#1)+/d 0.2cm/^{#3}}
\newcommand{\dltwocell}[3][0.5]{\ar@{}[#2] \ar@{=>}?(#1)+/ur  0.2cm/;?(#1)+/dl 0.2cm/^{#3}}
\newcommand{\drtwocell}[3][0.5]{\ar@{}[#2] \ar@{=>}?(#1)+/ul  0.2cm/;?(#1)+/dr 0.2cm/^{#3}}
\newcommand{\dthreecell}[3][0.5]{\ar@{}[#2] \ar@3{->}?(#1)+/u  0.2cm/;?(#1)+/d 0.2cm/^{#3}}
\newcommand{\utwocell}[3][0.5]{\ar@{}[#2] \ar@{=>}?(#1)+/d 0.2cm/;?(#1)+/u 0.2cm/_{#3}}
\newcommand{\dtwocelltarg}[3][0.5]{\ar@{}#2 \ar@{=>}?(#1)+/u  0.2cm/;?(#1)+/d 0.2cm/^{#3}}
\newcommand{\utwocelltarg}[3][0.5]{\ar@{}#2 \ar@{=>}?(#1)+/d  0.2cm/;?(#1)+/u 0.2cm/_{#3}}

\newcommand{\pullbackcorner}[1][dr]{\save*!/#1+1.2pc/#1:(1,-1)@^{|-}\restore}

\newdir{(}{{}*!<0em,-.14em>-\cir<.14em>{l^r}}
\newdir{ (}{{}*!/-5pt/\dir{(}}
\newdir{ >}{{}*!/-5pt/\dir{>}}

\theoremstyle{plain}
\newtheorem{Thm}{Theorem}[section]
\newtheorem{Prop}[Thm]{Proposition}
\newtheorem{Cor}[Thm]{Corollary}

\theoremstyle{definition}
\newtheorem{Defn}[Thm]{Definition}

\theoremstyle{remark}
\newtheorem{Ex}[Thm]{Example}

\begin{document}
\leftmargini=2em
\title{Homomorphisms of higher categories}
\author{Richard Garner}
\address{Department of Pure Mathematics and Mathematical Statistics, University of Cambridge, Cambridge CB3 0WB, UK}
\email{rhgg2@cam.ac.uk} 
\date{\today}
\begin{abstract}
We describe a construction that to each algebraically specified notion of
higher-dimensional category associates a notion of \emph{homomorphism} which
preserves the categorical structure only up to weakly invertible higher cells.
The construction is such that these homomorphisms admit a strictly associative
and unital composition. We give two applications of this construction. The
first is to tricategories; and here we do not obtain the trihomomorphisms
defined by Gordon, Power and Street, but rather something which is equivalent
in a suitable sense. The second application is to Batanin's weak
$\omega$-categories.
\end{abstract}
 \maketitle

\section{Introduction}
The purpose of this paper is to describe a notion of
homomorphism for weak higher-dimensional categories. Let us at
once say that we concern ourselves exclusively with those
notions of higher-dimensional category which are
essentially-algebraic in the sense described by
Freyd~\cite{Freyd1972Aspects}; for which composition and its
associated coherence are realised by specified operations
subject to equational laws. Of course any species of
essentially-algebraic structure has a concomitant notion of
homomorphism, given by functions on the underlying data
commuting to the specified operations: but it is a commonplace
that for higher-dimensional categories, such homomorphisms are
too strict to be of practical use (though they retain
significant \emph{theoretical} importance), because they must
preserve the categorical structure ``on-the-nose'' rather than
up to suitably coherent higher cells. It is this latter, looser
notion of homomorphism that we shall concern ourselves with
here.

In low dimensions, the homomorphisms we seek already have
satisfactory descriptions: in the case of bicategories, they
are B\'enabou's homomorphisms~\cite[\S
4]{B'enabou1967Introduction}, whilst for tricategories we have
the \emph{trihomomorphisms} of~\cite[\S
3]{Gordon1995Coherence}. This gives us little direct insight
into how the general case should look; yet there is a
particular aspect of the low-dimensional examples which can
usefully be incorporated into a general theory, namely the idea
that, as important as the homomorphisms are, of greater
importance still is their relationship with the \emph{strict}
homomorphisms---the maps we described earlier as preserving the
categorical structure ``on-the-nose''. In the case of
bicategories, this relationship is described by the
two-dimensional monad theory
of~\cite{Blackwell1989Two-dimensional}. We write $\cat{CatGph}$
for the $2$-category of $\cat{Cat}$-enriched graphs---whose
objects are given by a set $X$ together with a functor $X
\times X \to \cat{Cat}$---and $T$ for the $2$-monad thereupon
whose algebras are small bicategories. There now arise both the
category $T\text-\cat{Alg}_{\mathrm s}$ of $T$-algebras and
strict $T$-algebra morphisms---which is equally well the
category $\cat{Bicat}_{\mathrm s}$ of bicategories and strict
homomorphisms---and also the category $T\text-\cat{Alg}$ of
$T$-algebras and $T$-algebra pseudomorphisms---which is equally
well (after some work) the category $\cat{Bicat}$ of
bicategories and homomorphisms (of course, each of these
categories has additional $2$-dimensional structure; but we
will not concern ourselves with that here). Theorem~3.13
of~\cite{Blackwell1989Two-dimensional} now describes the
fundamental relationship between these two categories in terms
of an adjunction:
\begin{equation}
\label{fundamentaladjunction}
\cd[@C+1em]{
    \cat{Bicat}_{\mathrm s} \ar@<5pt>@{<-}[r]^-{(\thg)'} \ar@{}[r]|-{\bot} &
    \cat{Bicat} \ar@<5pt>@{<-}[l]^-{J}
}
\end{equation}
where $J$ is the identity-on-objects inclusion functor. The force of this is
that homomorphisms $\A \to \B$ are classified by strict homomorphisms $\A' \to
\B$, so that the seemingly inflexible strict homomorphisms are in fact the more
general notion. The adjunction in~\eqref{fundamentaladjunction} is of
fundamental importance to the theory developed
in~\cite{Blackwell1989Two-dimensional}, and a suitable generalisation of it
seems a natural desideratum for a theory of higher-dimensional homomorphisms.

Let us examine the ramifications of incorporating such a
generalisation into our theory. Suppose we are presented with
some notion of higher-dimensional category: in accordance with
our assumptions, it admits an essentially-algebraic
presentation, and as such we have a notion of strict
homomorphism, giving us the morphisms of a category
$\cat{HCat}_\mathrm s$. We wish to find the remaining elements
of~\eqref{fundamentaladjunction}: thus a category $\cat{HCat}$
whose maps are the homomorphisms and an adjunction $(\thg)'
\dashv J \colon \cat{HCat}_\mathrm s \to \cat{HCat}$ in which
$J$ is the identity on objects. Now to give these data is
equally well to give a comonad on $\cat{HCat}_s$, since on the
one hand, any adjunction $(\thg)' \dashv J$ of the required
form determines a comonad $(\thg)' \circ J$ on $\cat{HCat}_s$;
and on the other, any comonad $\mathsf Q$ on $\cat{HCat}_s$
determines an adjunction of the required form upon taking
$\cat{HCat}$ to be the \emph{co-Kleisli category} of $\mathsf
Q$ (whose definition we recall in Section~\ref{basicdef}
below).
Consequently, we can restate the problem of defining a notion
of homomorphism in terms of that of defining a suitable comonad
on the category of strict homomorphisms.

One technique for constructing such a comonad is suggested
in~\cite{Hess2006Co-rings}. For this we must suppose the
category $\cat{HCat}_\mathrm s$ to be presentable as the
category of algebras for a symmetric operad $\O$ on a suitable
base category $\V$; and may then consider \emph{co-rings} over
the operad $\O$---these being $\O$-$\O$-bimodules equipped with
comonoid structure in the monoidal category of
$\O$-$\O$-bimodules. Each such co-ring $\M$ induces a comonad
$\M \otimes_\O \thg$ on $\cat{HCat}_\mathrm s$, and hence a
notion of homomorphism. The problem with this approach lies in
the initial supposition of operadicity; which though it may be
appropriate for homological algebra is rather infrequently
satisfied in the case of higher categories. We may try and
rectify this by moving from symmetric operads to the higher
operads of Batanin~\cite{Batanin1998Monoidal}; but here a
different problem arises, namely that the tensor product of
bimodules over a globular operad is ill-defined, for the reason
that, in the category whose monoids are globular operads, the
tensor product does not preserve reflexive coequalisers in both
variables. Thus, though one can speak of bimodules---as Batanin
himself does in~\cite[Definition
8.8]{Batanin1998Monoidal}---one cannot speak of co-rings: and
so the homomorphisms we obtain need not admit a composition.

\looseness=-1 In this paper, we adopt a quite different means
of constructing a comonad on the category of strict
homomorphisms, one informed by categorical homotopy theory.
Lack, in~\cite{Lack2004Quillen}, establishes that the comonad
on $\cat{Bicat}_\mathrm s$ generated by the adjunction
in~\eqref{fundamentaladjunction} gives a notion of
\emph{cofibrant replacement} for a certain Quillen model
structure on $\cat{Bicat}_\mathrm s$; whose generating
cofibrations are the inclusions of the basic $n$-dimensional
boundaries into the basic $n$-dimensional cells. For the
general case, we can run this argument backwards: given a
Quillen model structure on $\cat{HCat}_\mathrm s$, we can---by
the machinery of~\cite{Garner2008Understanding}---use it to
generate a ``cofibrant replacement comonad'', and so obtain a
notion of homomorphism. In fact, to generate a cofibrant
replacement comonad we do not need a full model structure on
$\cat{HCat}_\mathrm s$, but only a single weak factorisation
system; and for this it suffices to give a set of generating
cofibrations, which as in the bicategorical case will be given
by the inclusions of $n$-dimensional boundaries into
$n$-dimensional cells.

%

The plan of the paper is as follows. We begin in
Section~\ref{basicdef} by giving a detailed explanation of the
general approach outlined above. We then give two applications.
The first, in Section~\ref{tricatdef}, is to the tricategories
of~\cite{Gordon1995Coherence}. In this case it may seem
redundant to define a notion of homomorphism, since as noted
above there is already one in the literature. However, the
homomorphisms that we define are better-behaved: they form a
category whereas the trihomomorphisms
of~\cite{Gordon1995Coherence} form, at best, a bicategory
(see~\cite{Garner2008low-dimensional} for the details). Now
this may lead us to question whether our homomorphisms are in
fact sufficiently weak. In order to show that they are, we
devote Section~\ref{s4} to a demonstration that the two
different notions of homomorphism, though not strictly the
same, are at least equivalent in a bicategorical sense. With
this as justification, we then give in Section~\ref{omcat} the
main application of our theory, to the definition of
homomorphisms between the weak $\omega$-categories of Michael
Batanin~\cite{Batanin1998Monoidal}.

\textbf{Acknowledgements}. The author thanks the organisers of PSSL~85, Nice,
and of Category Theory 2008, Calais, at which material from this paper was
presented, and an anonymous referee for a number of useful suggestions for improvement. He also acknowledges the support of a Research Fellowship of St John's College, Cambridge and of a Marie Curie Intra-European Fellowship, Project No.\ 040802.

\section{Homotopy-theoretic framework}\label{basicdef}
We saw in the Introduction that in order to obtain a notion of
homomorphism for some essentially-algebraic notion of
higher-dimensional category, it suffices to generate a suitable
comonad $\mathsf Q = (Q, \Delta, \epsilon)$ on the category
$\cat{HCat}_\mathrm s$ of strict homomorphisms: for then we may
then \emph{define} a homomorphism from $A$ to $B$ to be a
strict homomorphism $QA \to B$. Moreover, we may compose two
such homomorphisms $f \colon QA \to B$ and $g \colon QB \to C$
according to the formula
\begin{equation*}
    QA \xrightarrow{\Delta_A} QQA \xrightarrow{Qf} QB \xrightarrow{g} C\ \text,
\end{equation*}
and, from the comonad laws, see that this composition is
associative and has identities given by the counit maps
$\epsilon_A \colon QA \to A$. Thus we obtain a category
$\cat{HCat}$ of homomorphisms: it is the \emph{co-Kleisli
category} of the comonad $\mathsf Q$.

The purpose of this Section is to describe how we may obtain
suitable comonads by taking cofibrant replacements for a weak
factorisation system on the category $\cat{HCat}_\mathrm s$. As
motivation, we first show how any weak factorisation system on
a category gives rise to the data (though not necessarily the
axioms) for a comonad. We recall
from~\cite{Bousfield1977Constructions} that a \emph{weak
factorisation system}, or \emph{w.f.s.},
 $(\L, \R)$ on a category~$\C$
is given by two classes $\L$ and~$\R$ of morphisms in~$\C$
which are each closed under retracts when viewed as full
subcategories of the arrow category $\C^\mathbf 2$, and which
satisfy the two axioms of \emph{factorisation}---that each $f
\in \C$ may be
    written as $f = pi$ where $i \in \L$ and $p \in \R$---%
    and
\emph{lifting}---that for each $i \in \L$ and $p \in
    \R$, we have $i \wo p$,
where to say that $i \wo p$ holds is to say that for each
commutative square
\begin{equation*}
    \cd{
      U \ar[r]^{f} \ar[d]_{i} &
      W \ar[d]^{p} \\
      V \ar[r]_{g}  &
      X
    }
\end{equation*}
we may find a filler $j \colon V \to W$ satisfying $ji = f$ and
$pj = g$. If $(\L, \R)$ is a w.f.s., then its two classes
determine each other via the formulae
\begin{align*}
 \R = \L^{\wo} &\defeq \set{g \in \mathrm{mor}\ \C}{f \wo g \text{ for all $f
\in \L$}}\\
\text{and} \quad \L =
{}^{\wo}\R &\defeq \set{f \in \mathrm{mor}\ \C}{f \wo g \text{ for all $g
\in \R$}}\ \text.
\end{align*}
For those weak factorisation systems that we will be
considering, the following terminology will be appropriate: the
maps in $\L$ we call \emph{cofibrations}, and the maps in $\R$,
\emph{acyclic fibrations}. Supposing $\C$ to have an initial
object $0$, we say that $U \in \C$ is \emph{cofibrant} just
when the unique map $0 \to U$ is a cofibration; and define a
\emph{cofibrant replacement} for $X \in \C$ to be a cofibrant
object $Y$ together with an acyclic fibration $p \colon Y \to
X$. The factorisation axiom implies that every $X \in \C$ has a
cofibrant replacement, obtained by factorising the unique map
$0 \to X$. Suppose now that for every $X$ we have made a choice
of such, which we denote by $\epsilon_X \colon QX \to X$; then
by the lifting axiom, for every $f \colon X \to Y$ in $\C$
there exists a filler for the square on the left, and for every
$X \in \C$ a filler for the square on the right of the
following diagram:
\begin{equation}\label{fillers}
    \cd{
        0 \ar[r]^{!} \ar[d]_{!} & QY \ar[d]^{\epsilon_Y} \\
        QX \ar[r]_{f.\epsilon_X} \ar@{.>}[ur] & X
    } \qquad \text{and} \qquad     \cd{
        0 \ar[r]^{!} \ar[d]_{!} & QQX \ar[d]^{\epsilon_{QX}} \\
        QX \ar[r]_{1_{QX}} \ar@{.>}[ur] & QX\rlap{ .}
    }
\end{equation}
If we now suppose choices of such fillers to have been
made---which we denote by $Qf \colon QX \to QY$ and $\Delta_X
\colon QX \to QQX$ respectively---then we see that we have
obtained all of the data required for a comonad $(Q, \epsilon,
\Delta)$. However, because these data have been chosen
arbitrarily, there is no reason to expect that the
coassociativity and counit axioms should hold, that $\Delta$
should be natural in $X$, or even that the assignation $f
\mapsto Qf$ should be functorial. Whilst in general we cannot
resolve these issues, we may do so for a large class of
w.f.s.'s, including those which in the sequel will interest us.

Recall that a w.f.s.\ is called \emph{cofibrantly generated} by
a set $J \subseteq \L$ if $\R = J^{\wo}$. The principal
technique by which we build cofibrantly generated w.f.s.'s is
the \emph{small object argument} of Quillen \cite[\S
II.3]{Quillen1967Homotopical} and Bousfield
\cite{Bousfield1977Constructions}, which tells us that if $\C$
is a cocomplete category, and $J$ a set of maps in it
satisfying a suitable smallness property, then there is a
w.f.s.\ $(\L, \R)$ on $\C$ given by $\R = J^{\wo}$ and $\L =
{}^{\wo}\R$. These hypotheses are most easily satisfied if $\C$
is a \emph{locally finitely presentable} (\emph{l.f.p.})
category---which is to say that it may be presented as the
category of models for an essentially-algebraic theory, or
equally well, the category of finite-limit preserving functors
$\M \to \cat{Set}$ for some finitely complete small category
$\M$. In this case, $\C$ is certainly cocomplete, and moreover
any set of maps $J$ in it will satisfy the required smallness
property, and so generate a w.f.s.\ on $\C$.

Let us now define a \emph{cofibrant replacement comonad} for a
w.f.s.\ $(\L, \R)$ to be a comonad $\mathsf Q = (Q, \epsilon,
\Delta)$ such that for each $X \in \C$, the map $\epsilon_X
\colon QX \to X$ provides a cofibrant replacement for $X$.

\begin{Prop}
If $\C$ is a l.f.p.\ category, and $J$ a set of maps in it,
then the w.f.s.\ $(\L, \R)$ cofibrantly generated by $J$ may be
equipped with a cofibrant replacement comonad.
\end{Prop}

\begin{proof}
By examination of the construction used in the small object
argument, we see that it provides a choice of $(\L,
\R)$-factorisation
\begin{equation}\label{choosefacs}
    X \xrightarrow f Y \qquad \mapsto \qquad X \xrightarrow{\lambda_f} Pf \xrightarrow{\rho_f} Y \qquad \text{(for all $f \in \C$)}
\end{equation}
that is \emph{functorial}, in the sense that it provides the
assignation on objects of a functor $\C^{\mathbf 2} \to
\C^{\mathbf 3}$ which is a section of the ``composition''
functor $\C^\mathbf 3 \to \C^\mathbf 2$. In particular, by
fixing $X$ to be $0$, we obtain a choice of cofibrant
replacements $\epsilon_Y \colon QY \to Y$ and of fillers $Qf
\colon QY \to QZ$ such that $f \mapsto Qf$ is a functorial
assignment and $\epsilon$ a natural transformation. It remains
only to construct natural maps $\Delta_Y \colon QY \to QQY$ for
which the comonad laws are satisfied, and this is done by
Radulescu-Banu in \cite[\S 1.1]{Radulescu-Banu1999Cofibrance};
we omit the details.
\end{proof}

In principle, we could end this section here, since we have now
shown how to associate a cofibrant replacement comonad to any
(well-behaved) category equipped with a (well-behaved) w.f.s.
However, there is something unsatisfactory about the previous
Proposition. An examination of its proof shows that a
cofibrantly generated w.f.s.\ $(\L, \R)$ may well admit many
different cofibrant replacement comonads, since the given
construction relies on arbitrary choices of data which, in
general, will induce non-isomorphic choices of $Q$. Firstly, we
must choose a generating set $J$ for $(\L, \R)$; and secondly,
we must choose a (sufficiently large) regular cardinal $\kappa$
that governs the length of the transfinite induction used in
the application of the small object argument. The first of
these choices should not worry us unduly, since in practice, it
is the set $J$ that one starts from, rather than the w.f.s.\ it
generates. However, the second is a more substantial concern,
since the piece of data on which it is predicated is one that
ought to remain entirely internal to the workings of the small
object argument. This raises the question as to whether there
is a canonical---or better yet, \emph{universal}---choice of
cofibrant replacement comonad associated to a w.f.s.\ $(\L,
\R)$. We will now show that there is, at least once we have
fixed a generating set $J$. To do so we will need to recall
some definitions from~\cite{Garner2008Understanding}.

\begin{Defn}
Let $(\L, \R)$ be a w.f.s.\ on a category $\C$. An
\emph{algebraic realisation} of $(\L, \R)$ is given by the
following pieces of data:
\begin{itemize}
\item For each $f \colon X \to Y$ in $\C$, a choice of
    $(\L, \R)$ factorisation as in~\eqref{choosefacs};
\item For each commutative square as on the left of the
    following diagram, a choice of filler as on the right:
\begin{equation*}
    \cd{
      U \ar[r]^{h} \ar[d]_{f} &
      W \ar[d]^{g} \\
      V \ar[r]_{k} &
      X
    } \qquad \dashrightarrow \qquad
    \cd{
      U \ar[r]^{\lambda_g. h} \ar[d]_{\lambda_f} &
      Pg \ar[d]^{\rho_g} \\
      Pf \ar[r]_{k . \rho_f} \ar@{.>}[ur]|{P(h, k)} &
      X\rlap{ ;}
    }
\end{equation*}
\item For each $f \colon X \to Y$ in $\C$, choices of
    fillers for the following squares:
\begin{equation*}
    \cd{
      X \ar[r]^{\lambda_{\lambda_f}} \ar[d]_{\lambda_f} &
      P\lambda_f \ar[d]^{\rho_{\lambda_f}} \\
      Pf \ar[r]_{1_{Pf}} \ar@{.>}[ur]|{\sigma_f} &
      Pf
    } \qquad \text{and} \qquad
    \cd{
      Pf \ar[r]^{1_{Pf}} \ar[d]_{\lambda_{\rho_f}} &
      Pf \ar[d]^{\rho_f} \\
      P\rho_f \ar[r]_{\rho_{\rho_f}} \ar@{.>}[ur]|{\pi_f} &
      Y
    }
\end{equation*}
\end{itemize}
subject to the following axioms:
\begin{itemize}
\item The assignation $f \mapsto \lambda_f$ is the functor
    part of a comonad $\mathsf L$ on $\C^\mathbf 2$ whose
    counit map at $f$ is $(1, \rho_f) \colon \lambda_f \to
    f$ and whose comultiplication is $(1, \sigma_f) \colon
    \lambda_f \to \lambda_{\lambda_f}$;
\item The assignation $f \mapsto \rho_f$ is the functor
    part of a monad $\mathsf R$ on $\C^\mathbf 2$ whose
    unit map at $f$ is $(\lambda_f,1) \colon f \to \rho_f$
    and whose multiplication is $(\pi_f, 1) \colon
    \lambda_f \to \lambda_{\lambda_f}$;
\item The natural transformation $LR \Rightarrow RL \colon
    \C^\mathbf 2 \to \C^\mathbf 2$ whose component at $f$
    is $(\sigma_f, \pi_f) \colon \lambda_{\rho_f} \to
    \rho_{\lambda_f}$ describes a distributive law in the
sense of~\cite{Beck1969Distributive} between $\mathsf L$
and $\mathsf R$.
\end{itemize}
\end{Defn}
The data for an algebraic realisation is sufficient to
reconstruct the underlying w.f.s.\ $(\L, \R)$: indeed, the
classes $\L$ and $\R$ are the closure under retracts of the
classes of maps admitting an $\mathsf L$-coalgebra structure,
respectively an $\mathsf R$-algebra, structure. Hence the pairs
$(\mathsf L, \mathsf R)$ arising from algebraic realisations
are objects worthy of study on their own: they are the
\emph{natural weak factorisation systems} of
\cite{Grandis2006Natural}. Note that the \emph{data} for an
algebraic realisation will exist for any weak factorisation
system; the issue is whether or not we may choose it such that
the \emph{axioms} are satisfied. The main result
of~\cite{Garner2008Understanding} is to show that for a
cofibrantly generated w.f.s., we can, and moreover, that there
is a best possible way of doing so.
\begin{Prop}\label{universalrealisation}
Let $\C$ be a l.f.p.\ category, and let $J$ be a set of maps in
it. Then the w.f.s.\ $(\L, \R)$ cofibrantly generated by $J$
has a universally determined algebraic realisation.
\end{Prop}
The sense of this universality is discussed in detail
in~\cite[\S 3]{Garner2008Understanding}; in brief, it says that
the universal algebraic realisation $(\mathsf L, \mathsf R)$ is
freely generated by the requirement that each map $j \in J$
should come equipped with a distinguished structure of $\mathsf
L$-coalgebra. In other words, given any other natural w.f.s.\
$(\mathsf L', \mathsf R')$ on $\C$ and a distinguished $\mathsf
L'$-coalgebra structure on each $j \in J$, we can find a unique
morphism of natural w.f.s.'s (see \cite[\S
3.3]{Garner2008Understanding} for the definition) $(\mathsf L,
\mathsf R) \to (\mathsf L', \mathsf R')$ preserving the
distinguished coalgebras. Note in particular that this
universality is determined \emph{by the set $J$}, and not
merely by the w.f.s.\ it generates; but as we have remarked
before, this should not worry us unduly, since in practice it
is the set $J$, rather than the w.f.s., from which one starts.

\begin{proof}[Proof of Proposition~\ref{universalrealisation}]
For a full proof see \cite[Theorem
4.4]{Garner2008Understanding}: we recall only the salient
details here. To construct the factorisation of a map $f \colon
X \to Y$ of $\C$, we begin exactly as in the small object
argument. We form the set $S$ whose elements are squares
\begin{equation*}
\cd{
    A \ar[r]^{h} \ar[d]_j & X \ar[d]^f \\
    B \ar[r]_k & Y
}
\end{equation*}
such that $j \in J$. We then form the coproduct
\begin{equation*}
    \cd[@C+2em]{
    \sum_{x \in S} A_x \ar[d]_{\sum_{x \in S} j_x} \ar[r]^-{\spn{h_x}_{x \in S}} &
    X \ar[d]^f \\
    \sum_{x \in S} B_x \ar[r]_-{\spn{h_x}_{x \in S}} &
    Y}
\end{equation*}
and define an object $P'g$ and morphisms $\lambda'_g$ and
$\rho'_g$ by factorising this square as
\begin{equation*}
    \cd[@C+2em]{
    \sum_{x \in S} A_x \ar[d]_{\sum_{x \in S} j_x} \ar[r]^-{\spn{h_x}_{x \in S}} &
    X \ar[d]^{\lambda'_f} \ar[r]^{\id_X} & X \ar[d]^f \\
    \sum_{x \in S} B_x \ar[r] &
    P'f \ar[r]_{\rho'_f} & Y}
\end{equation*}
where the left-hand square is a pushout. The assignation $f
\mapsto \rho'_f$ may now be extended to a functor $R' \colon
\C^\mathbf 2 \to \C^\mathbf 2$; whereupon the map $(\lambda'_f,
\id_Y) \colon f \to R'f$ provides the component at $f$ of a
natural transformation $\Lambda' \colon \id_{\C^\mathbf 2}
\Rightarrow R'$. We now obtain the monad part $\mathsf R$ of
the desired algebraic realisation as the free monad on the
pointed endofunctor $(R', \Lambda')$. We may construct this
using the techniques of~\cite{Kelly1980unified}. To obtain the
comonad part $\mathsf L$ we proceed as follows. The assignation
$f \mapsto \lambda'_f$ underlies a functor \mbox{$L' \colon
\C^\mathbf 2 \to \C^\mathbf 2$}; and a little manipulation
shows that this functor in turn underlies a comonad
$\mathsf{L}'$ on $\C^\mathbf 2$. We may now adapt the free
monad construction so that at the same time as it produces
$\mathsf R$ from $(R', \Lambda')$, it also produces $\mathsf L$
from~$\mathsf L'$.
\end{proof}

\begin{Cor}\label{universalcofib}
Let $\C$ be a l.f.p.\ category, and let $J$ be a set of maps in
it. Then the w.f.s.\ $(\L, \R)$ generated by $J$ may be
equipped with a universally determined cofibrant replacement
comonad.
\end{Cor}
\begin{proof}
Form the universal algebraic realisation $(\mathsf L, \mathsf
R)$ of $J$; now define the universal cofibrant replacement
comonad to be the restriction of the comonad $\mathsf L$ to the
coslice category $0 / \C \cong \C$.
\end{proof}

The preceding proofs provide us with a very general machinery
for building the universal cofibrant replacement comonad of a
w.f.s. In practice, however, it is often easier to describe
directly what we think this comonad should be; and so we now
give a recognition principle that will allow us to prove such a
description to be correct.

\begin{Defn}
Let $J$ be a fixed set of maps in a category $\C$. Now for any
$f \colon Y \to X$ in $\C$, a \emph{choice of liftings} for $f$
(with respect to $J$) is a function $\phi$ which to every $j
\in J$ and commutative square
\begin{equation}\label{commsquare}
    \cd{
        A \ar[r]^h \ar[d]_j & Y \ar[d]^f \\
        B \ar[r]_k \ar@{.>}[ur] & X
    }
\end{equation}
in $\C$ assigns a diagonal filler $\phi(j, h, k) \colon B \to
Y$ making both triangles commutate as indicated. We call the
pair $(f, \phi)$ an \emph{algebraic acyclic fibration}. Given
an object $X \in \C$, we define the category $\cat{AAF}/X$ to
have as objects, algebraic acyclic fibrations into $X$, and as
morphisms $(f, \phi) \to (g, \psi)$, commutative triangles
\begin{equation*}
    \cd{
        Y \ar[dr]_{f} \ar[rr]^u & & Z \ar[dl]^{g} \\ & X
    }
\end{equation*}
such that for any square of the form~\eqref{commsquare} we have
$u.\phi(j,h,k) = \psi(j, uh, k)$.
\end{Defn}
Our recognition principle is now the following:
\begin{Prop}\label{recognition}
Let $J$ be a set of maps in a l.f.p.\ category $\C$. Then for
each $X \in \C$, the universal cofibrant replacement $\epsilon_X
\colon QX \to X$ with respect to $J$ may be equipped with a
choice of liftings $\phi_X$ such that $(\epsilon_X, \phi_X)$
becomes an initial object of $\cat{AAF} / X$.
\end{Prop}
\begin{proof}
Let $(\mathsf L, \mathsf R)$ be the universal algebraic
realisation of $J$. It follows from~\cite[Proposition
5.4]{Garner2008Understanding} that $\cat{AAF}/X$ is isomorphic
to the category of algebras for the monad $\mathsf R_X$
obtained by restricting and corestricting the monad $\mathsf R
\colon \C^\mathbf 2 \to \C^\mathbf 2 $ to the slice category
$\C / X$. As such, it has an initial object obtained by
applying the free functor $\C / X \to \cat{AAF} / X$ to the
initial object $0 \to X$ of $\C / X$. Moreover, the underlying
map of this initial object is obtained by applying $R$ to $0
\to X$, and hence is the universal cofibrant replacement
$\epsilon_X \colon QX \to X$.
\end{proof}

\begin{Ex}\label{ex1}
Let $S$ be a commutative ring, and consider the category
$\cat{Ch}(S)$ of positively graded chain complexes of
$S$-modules, equipped with the set of generating cofibrations
$J := \set{2_i \hookrightarrow \partial_i}{i \in \mathbb N}$.
Here $2_i$ is the representable chain complex at $i$, with
components given by
\begin{equation*}
    (2_i)_n = \begin{cases} S & \text{if $n = i$ or $n = i-1$;} \\ 0 &
    \text{otherwise,}\end{cases}
\end{equation*}
and differential being the identity map at stage $i$ and the
zero map elsewhere. The chain complex $\partial_i$ is its
boundary, whose components are
\begin{equation*}
    (\partial_i)_n = \begin{cases} S & \text{if $n = i-1$;} \\ 0 &
    \text{otherwise,}\end{cases}
\end{equation*}
and whose differential is everywhere zero. $\cat{Ch}(S)$ is a
l.f.p.\ category, and so by Corollary~\ref{universalcofib} we
may take universal cofibrant replacements with respect to $J$.
We now give an explicit description of these cofibrant
replacements. Given a chain complex $X$, the chain complex $QX$
will be free in every dimension; and so it suffices to specify
a set of free generators for each $(QX)_i$ and to specify where
each generator should be sent by the differential $d_i \colon
(QX)_i \to (QX)_{i-1}$ and the counit $\epsilon_i \colon (QX)_i
\to X_i$. We do this by induction over $i$:
\begin{itemize}
\item For the base step, $(QX)_0$ is generated by the set
    $\set{[x]}{x \in X_0}$, and $\epsilon_0 \colon (QX)_0
    \to X_0$ is specified by $\epsilon_0([x]) = x$;
\item For the inductive step, $(QX)_{i+1}$ (for $i
    \geqslant 0$) is generated by the set
\begin{equation*}
    \set{[x, z]}{x \in X_{i+1},\ z \in Z(QX)_i,\ \epsilon_i(z) =
d_{i+1}(x)}\text,
\end{equation*}
whilst $\epsilon_{i+1} \colon (QX)_{i+1} \to X_{i+1}$ and
$d'_{i+1} \colon (QX)_{i+1} \to (QX)_i$ are specified by $
\epsilon_{i+1}([x, z]) = x$ and $d_{i+1}([x, z]) = z$,
\end{itemize}
where given a chain complex $A$, we are writing $ZA_i$ for the
kernel of the map $d_i \colon A_i \to A_{i-1}$. To prove that
$\epsilon_X$ is the universal cofibrant replacement for $X$, it
suffices, by Proposition~\ref{recognition}, to equip it with a
choice of liftings such that it becomes an initial object of
$\cat{AAF} / X$. By inspection, to equip a chain map $f \colon
Y \to X$ with a choice of liftings is to give:
\begin{itemize}
\item A set function $k_0 \colon X_0 \to Y_0$ which is a
    section of $f_0 \colon Y_0 \to X_0$;
\item For every $i \geqslant 0$, a set function $k_{i+1}
    \colon X_{i+1} \times ZY_i \to Y_{i+1}$ which is a
    section of $(f_{i+1}, d_{i+1}) \colon Y_{i+1} \to
    X_{i+1} \times ZY_i$.
\end{itemize}
The map $\epsilon_X \colon QX \to X$ has an obvious choice of
liftings given by the inclusion of generators. We claim that
this makes it initial in $\cat{AAF}/X$. Indeed, given $f \colon
Y \to X$ equipped with a choice of liftings $\{k_i\}$, there is
a chain map $h \colon QX \to Y$ given by the following
recursion:
\begin{itemize}
\item For the base step, $h_0$ is specified by $h_0([x]) =
    k_0(x)$;
\item For the inductive step, $h_{i+1}$ is specified by
    $h_{i+1}([x,z]) = k_{i+1}(x,h_i(z))$.
\end{itemize}
It's easy to see that this $h$ commutes with the projections to
$X$, and with the given choices of liftings; and moreover, that
it is the unique chain map $QX \to Y$ with these properties.
Hence, by Proposition~\ref{recognition}, $\epsilon_X \colon QX
\to X$ is the universal cofibrant replacement of $X$.
\end{Ex}

Now, although Proposition~\ref{recognition} allows us to
recognise the functor and the counit part of the universal
cofibrant replacement comonad, it says nothing about its
comultiplication. In fact, we may recover this using the
initiality exhibited in Proposition~\ref{recognition}. We first
observe that if $f \colon C \to D$ and $g
 \colon D \to E$ are equipped with choices of liftings $\phi$ and $\psi$,
then their composite $gf \colon C \to E$ may also be so
equipped, via the assignation $(\phi \bullet \psi)(j, h, k)
\defeq \phi(j, h, \psi(j, fh, k))$.

\begin{Prop}\label{inducecomult}
Let $J$ be a set of maps in a l.f.p.\ category $\C$. Then for
each $X \in \C$, the unique map
\begin{equation}\label{uniquemap}
    \cd{
        QX \ar@{.>}[rr] \ar[dr]_{(\epsilon_X, \phi_X)} & & QQX \ar[dl]^{(\epsilon_X . \epsilon_{QX}, \phi_X \bullet \phi_{QX})} \\ & X\rlap{ .}
    }
\end{equation}
of $\cat{AAF} / X$ is the comultiplication $\Delta_X \colon QX
\to QQX$ of the universal cofibrant replacement comonad
generated by $J$.
\end{Prop}
\begin{proof}
It suffices to check that $\Delta_X \colon QX \to QQX$
renders~\eqref{uniquemap} commutative, and that it respects the
chosen liftings. The first of these conditions follows from the
comonad axioms. For the second, we again make use of the
isomorphic between $\cat{AAF}/X$ and the category of algebras
for the monad $\mathsf R_X \colon \C/X \to \C/X$ obtained from
the universal algebraic realisation of $J$. To show that
$\Delta_X$ respects the chosen liftings in~\eqref{uniquemap} is
equally well to show that it respects the corresponding
$\mathsf R_X$-algebra structures on $\epsilon_X$ and
$\epsilon_X.\epsilon_{QX}$, and we now do so by explicit
calculation.

First let us introduce some notation: we write $f$ to denote
the unique map $0 \to X$ in $\C$. Now the map $\epsilon_X
\colon QX \to X$ is equally well the map $\rho_f \colon Pf \to
X$, and in these terms its $\mathsf R_X$-algebra structure is
the morphism
\begin{equation*}
    \cd{
        P\rho_f \ar[rr]^{\pi_f} \ar[dr]_{\rho_{\rho_f}} & & P f \ar[dl]^{\rho_f} \\ & X
    }
\end{equation*}
of $\C / X$. Likewise, the map $\epsilon_X . \epsilon_{QX}
\colon QQX \to X$ is equally well the map $\rho_f .
\rho_{\lambda_f} \colon P\lambda_f \to X$, in which terms its
$\mathsf R_X$-algebra structure will be given by a morphism $
\theta_f \colon P(\rho_f . \rho_{\lambda_f}) \to P\lambda_f$
over $X$. To describe this map we appeal to Theorem A.1
of~\cite{Garner2008Understanding}, which shows that it is given
by the following composite
\begin{equation*}
    P(\rho_f . \rho_{\lambda_f})
    \xrightarrow{\sigma_{\rho_f . \rho_{\lambda_f}}}
    P\lambda_{\rho_f . \rho_{\lambda_f}}
    \xrightarrow{P(1, P(\rho_{\lambda_f}, 1))}
    P(\lambda_{\rho_f}.\rho_{\lambda_f})
    \xrightarrow{P(1, \pi_f)}
    P\rho_{\lambda_f}
    \xrightarrow{\pi_{\lambda_f}}
    P \lambda_f\ \text.
\end{equation*}
Now, the map $\Delta_X \colon QX \to QQX$ is equally well the
map $\sigma_f \colon Pf \to P\lambda_f$, and so to check that
it is an $\mathsf R_X$-algebra map, and thereby complete the
proof, it suffices to show that the square
\begin{equation*}
    \cd[@C+1em]{
      P\rho_f \ar[r]^-{P(\sigma_f, 1)} \ar[d]_{\pi_f} & P(\rho_f.\rho_{\lambda_f}) \ar[d]^{\theta_f} \\
      Pf \ar[r]_-{\sigma_f} & P\lambda_f
    }
\end{equation*}
commutes; and this follows by a short calculation with the
axioms for a natural w.f.s.
\end{proof}

\begin{Ex}
We consider again the situation of Example~\ref{ex1}. Given a
chain complex $X$, the canonical choice of liftings for the map
$\epsilon_X . \epsilon_{QX} \colon QQX \to X$ is given as
follows:
\begin{itemize}
\item For the base step, $k_0 \colon X_0 \to (QQX)_0$ is
    given by $k_0(x) = [[x]]$;
\item For the inductive step, $k_{i+1} \colon X_{i+1}
    \times Z(QQX)_i \to (QQX)_{i+1}$ is given by
    $k_{i+1}(x, z) = [[x, \epsilon_{QX}(z)], z]$.
\end{itemize}
It follows from this, the description of the initiality of
$\epsilon_X$ given in Example~\ref{ex1}, and
Proposition~\ref{inducecomult}, that the comultiplication map
$\Delta_X \colon QX \to QQX$ has components given by the
following recursion:
\begin{itemize}
\item For the base step, $\Delta_0 \colon (QX)_0 \to
    (QQX)_0$ is specified by $\Delta_0([x]) = [[x]]$;
\item For the inductive step, $(\Delta_X)_{i+1} \colon
    (QX)_{i+1} \to (QQX)_{i+1}$  is specified by
    $\Delta_{i+1}([x,z]) = [[x, z],
\Delta_i(z)]$.
\end{itemize}
\end{Ex}

%

\section{Homomorphisms of tricategories}\label{tricatdef}
\looseness=-1 In the following Sections we give two
applications of the general theory described above. In the
present Section, we shall use it to develop a notion of
homomorphism between the \emph{tricategories}
of~\cite{Gordon1995Coherence}. We begin in \S\ref{tri1} by
defining a category $\cat{Tricat}_\mathrm s$ of tricategories
and strict homomorphisms, and distinguishing in it a suitable
set of generating cofibrations. Then in \S\ref{tri2} we
characterise the universal cofibrant replacement comonad this
generates; and finally in \S\ref{tri3}, we extract a concrete
description of the co-Kleisli category of this comonad, which
will be the desired category of trihomomorphisms.

Since there is already in the literature a notion of
trihomomorphism (see~\cite[\S 3]{Gordon1995Coherence}, for
instance), it is reasonable to ask why we should go to the
effort of defining another one. There are two main reasons to
do so. The first is that it illustrates the operation of our
machinery in a relatively elementary case, which will prove
useful in understanding the $\omega$-categorical application of
Section~\ref{omcat} below. The second is that the
trihomomorphisms we describe are better-behaved than the
existing ones: in particular, ours admit a \emph{strictly}
associative and unital composition.

Now, the fact that our trihomomorphisms are better-behaved
could suggest that they are insufficiently weak, and hence that
our general machinery is not fit for the task. In order to show
that this is not the case, we give in Section~\ref{s4} a
careful comparison between our trihomomorphisms and those
of~\cite{Gordon1995Coherence}, and show that the two are the
same in a suitable sense, by proving a biequivalence between
two bicategories whose $0$-cells are tricategories, and whose
$1$-cells are trihomomorphisms of the two different kinds.

\subsection{Generating cofibrations}\label{tri1}
The notion of tricategory was introduced
in~\cite{Gordon1995Coherence}, yet the formulation given there
is unsuitable for our purposes since it is not wholly
algebraic: it asserts certain morphisms of a hom-bicategory to
be equivalences without requiring choices of pseudo-inverse to
be provided. Instead we shall adopt\footnote{With one minor
alteration: we ask that the homomorphisms of bicategories $1
\to \T(x, x)$ picking out units should be normalised. This
change is not substantive, since any homomorphism of
bicategories can be replaced with a normal one; but it does
reduce slightly the amount of coherence data we have to deal
with.} the definition of~\cite{Gurski2006algebraic}, for which
such choices are part of the data.
\begin{Defn}
The category $\cat{Tricat}_\mathrm s$ has as objects,
tricategories in the sense
of~\cite[Chapter~4]{Gurski2006algebraic}; and as morphisms
$\mathcal T \to \mathcal U$, assignations on $0$-, $1$-, $2$-
and $3$-cells which commute with the tricategorical operations
on the nose.
\end{Defn}
We observe that $\cat{Tricat}_\mathrm s$ is the category of
models of an essentially-algebraic theory, and as such is
locally finitely presentable. Therefore we may use
Corollary~\ref{universalcofib} to describe a cofibrant
replacement comonad on it, as soon as we have distinguished in
it a suitable set of generating cofibrations. Before doing so,
we observe that underlying any tricategory is a
three-dimensional globular set; that is, a presheaf over the
category $\cat{G}_3$ generated by the graph
\begin{equation*}
    \cd{
    0 \ar@<3pt>[r]^-{\sigma} \ar@<-3pt>[r]_-{\tau} & 1 \ar@<3pt>[r]^-{\sigma} \ar@<-3pt>[r]_-{\tau} & 2 \ar@<3pt>[r]^-{\sigma} \ar@<-3pt>[r]_-{\tau} & 3}\ \text,
\end{equation*}
subject to the equations $\sigma \sigma = \tau \sigma$ and
$\sigma \tau = \tau \tau$. Thus there is a functor $V \colon
\cat{Tricat}_{\mathrm s} \to [{\cat G}_3^\op, \cat{Set}]$
which, because it is given by forgetting essentially-algebraic
structure, has a left adjoint $K \colon [{\cat G}_3^\op,
\cat{Set}] \to \cat{Tricat}_\mathrm s$.

\begin{Defn}\label{gencofibstricat}
The \emph{generating cofibrations} of $\cat{Tricat}_\mathrm s$
are the morphisms $\set{\iota_n \colon \partial_n \to 2_n}{0 \leqslant
n \leqslant 4}$ obtained by applying the functor $K$ to the
morphisms $f_0, \dots, f_4$ of $[\cat G_3^\op, \cat{Set}]$
given as follows (where we write $y$ for the Yoneda embedding
$\cat G_3 \to [\cat G_3^\op, \cat{Set}]$):
\begin{itemize}
\item $f_0$ is the unique map $0 \to y_0$;
\item $f_1$ is the map $[y_\sigma, y_\tau] \colon y_0 + y_0
    \to y_1$;
\item $f_2$ and $f_3$ are the maps induced by the universal
    property of pushout in the following diagram (for $n =
    2, 3$):
\begin{equation*}
    \cd[@C+1em]{
        y_{n-2} + y_{n-2} \ar[r]^-{[y_\sigma,y_\tau]} \ar[d]_{[y_\sigma,y_\tau]} &
        y_{n-1} \ar[d] \ar[ddr]^{y_\tau}\\
        y_{n-1} \ar[r] \ar[drr]_{y_\sigma} &
        \star \pullbackcorner \ar@{.>}[dr]|{f_{n}} \\
        & & y_{n}
    }
\end{equation*}
\item $f_4$ is the map induced by the universal property of
    pushout in the following diagram:
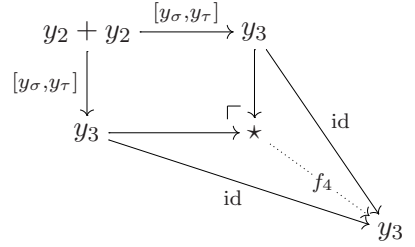
\begin{equation*}
    \cd[@C+1em]{
        y_2 + y_2 \ar[r]^-{[y_\sigma,y_\tau]} \ar[d]_{[y_\sigma,y_\tau]} &
        y_3 \ar[d] \ar[ddr]^{\id}\\
        y_3 \ar[r] \ar[drr]_{\id} &
        \star \pullbackcorner \ar@{.>}[dr]|{f_4} \\
        & & y_{3}\rlap{ .}
    }
\end{equation*}
\end{itemize}
In diagrammatic terms, $f_0, \dots, f_4$ are the following
maps: \vskip0.3\baselineskip
%
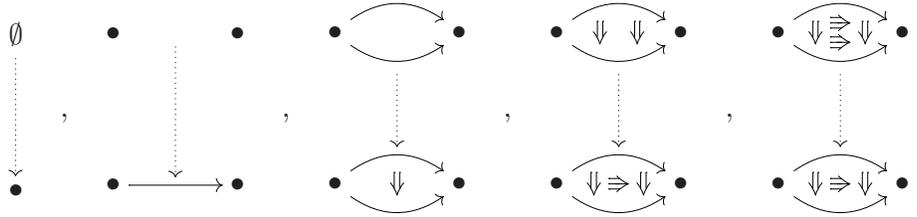
\begin{equation*}
\cd[@R+2em@C+1em]{\emptyset \ar@{.>}[d] \\ \bullet} \quad \text, \quad
\cd[@R+2em@C+1em]{ \bullet \ar@{}[r]_{}="a" & \bullet  \\
  \bullet \ar[r]^{}="b" \ar@{.>}"a"; "b"& \bullet}\quad \text, \quad
\cd[@R+2em@C+1em]{ \bullet \ar@/^1em/[r] \ar@/_1em/[r]_{}="a" & \bullet  \\
  \bullet \ar@/^1em/[r]^{}="b" \ar@/_1em/[r] \ar@{}[r]
  \ar@{=>}?(0.5)+/u  0.15cm/;?(0.5)+/d 0.15cm/ \ar@{.>}"a"; "b"& \bullet}\quad \text, \quad
\cd[@R+2em@C+1em]{ \bullet \ar@/^1em/[r] \ar@/_1em/[r]_{}="a" \ar@{}[r] \ar@{=>}?(0.65)+/u  0.15cm/;?(0.65)+/d 0.15cm/
  \ar@{}[r] \ar@{=>}?(0.35)+/u  0.15cm/;?(0.35)+/d 0.15cm/ & \bullet  \\
  \bullet \ar@{}[r] \ar@3?(0.5)+/l 0.13cm/;?(0.5)+/r 0.13cm/ \ar@/^1em/[r]^{}="b" \ar@/_1em/[r] \ar@{}[r]
  \ar@{=>}?(0.7)+/u  0.15cm/;?(0.7)+/d 0.15cm/ \ar@{}[r] \ar@{=>}?(0.3)+/u  0.15cm/;?(0.3)+/d 0.15cm/ &
  \bullet \ar@{.>}"a"; "b"}\quad \text, \quad
\cd[@R+2em@C+1em]{ \bullet \ar@{}[r] \ar@3?(0.5)+/l 0.13cm/+/u 0.13cm/;?(0.5)+/r 0.13cm/+/u 0.13cm/
  \ar@3?(0.5)+/l 0.13cm/+/d 0.26cm/;?(0.5)+/r 0.13cm/+/d 0.26cm/ \ar@/^1em/[r] \ar@/_1em/[r]_{}="a"
  \ar@{}[r] \ar@{=>}?(0.7)+/u  0.15cm/;?(0.7)+/d 0.15cm/ \ar@{}[r] \ar@{=>}?(0.3)+/u  0.15cm/;?(0.3)+/d 0.15cm/ & \bullet\\
  \bullet \ar@{}[r] \ar@3?(0.5)+/l 0.13cm/;?(0.5)+/r 0.13cm/ \ar@/^1em/[r]^{}="b" \ar@/_1em/[r] \ar@{}[r]
  \ar@{=>}?(0.7)+/u  0.15cm/;?(0.7)+/d 0.15cm/ \ar@{}[r] \ar@{=>}?(0.3)+/u  0.15cm/;?(0.3)+/d 0.15cm/ &
  \bullet\ar@{.>}"a"; "b"}\text.
\end{equation*}
\end{Defn}
\vskip\baselineskip
\begin{Defn}\label{tricatdef1} We define $\mathsf Q \colon \cat{Tricat}_\mathrm s \to \cat{Tricat}_\mathrm s$
to be the universal cofibrant replacement comonad for the
generating cofibrations of Definition~\ref{gencofibstricat},
and define the category $\cat{Tricat}$ of tricategories and
trihomomorphisms to be the co-Kleisli category of this comonad.
\end{Defn}

\subsection{Universal cofibrant replacement}\label{tri2}
The aim of this section is to obtain a concrete description of
the comonad~$\mathsf Q$. As in Example~\ref{ex1}, the easiest
way of doing this will not be to work through the construction
given in Proposition~\ref{universalrealisation}; rather, it
will be to describe directly the universal cofibrant
replacements and then prove our description correct by
appealing to Proposition~\ref{recognition}. In order to give
this description, we will need to develop some constructions on
tricategories. First we observe that any tricategory~$\T$ has
an underlying one-dimensional globular set, comprised of the
$0$- and $1$-cells of $\T$; and so we have an adjunction
\begin{equation*}
    L \dashv W \colon \cat{Tricat}_\mathrm s \to [\cat{G}_1^\op, \cat{Set}]
\end{equation*}
(where $\cat{G}_1$ is the category $1 \rightrightarrows 0$).
Given some $X \in [\cat{G}_1^\op, \cat{Set}]$, we may take $LX$
to have the same $0$-cells as $X$, and write $[f] \colon x \to
y$ for the image in $LX$ of a $1$-cell $f \colon x \to y$ of
$X$. We next describe what it means to \emph{adjoin a $2$-cell}
to a tricategory $\T$. Given a pair of parallel $1$-cells $f, g
\colon X \to Y$ in $\T$, there is a unique strict homomorphism
$(f,g) \colon \partial_2 \to \T$ sending the generating
$1$-cells of $\partial_2$ to $f$ and $g$ respectively. Since
$\cat{Tricat}_\mathrm s$ is locally finitely presentable, it is
in particular cocomplete, and so we may define a new
tricategory $\T[\alpha]$ by means of the following pushout:
\begin{equation}\label{adjoin2cell}
    \cd{
    \partial_2 \ar[d]_{\iota_2}
    \ar[r]^{(f,g)} &
    \T \ar[d]^{\eta} \\
    2_2 \ar[r]_{\overline \alpha} &
    \T[\alpha]\rlap{ .}
    }
\end{equation}
We say that $\T[\alpha]$ is obtained from $\T$ by adjoining a
$2$-cell $\alpha \colon f \Rightarrow g$. Indeed, to give a
strict homomorphism $F \colon \T[\alpha] \to \U$ is equally
well to give its restriction $F \eta \colon \T \to \U$ together
with the $2$-cell $Ff \Rightarrow Fg$ named by $F\bar \alpha
\colon D_2 \to \V$. By replacing the morphism $\iota_2$
in~\eqref{adjoin2cell} with a suitable coproduct of
$\iota_2$'s, we may extend this definition to deal with the
simultaneous adjunction to $\T$ of any set-sized collection of
$2$-cells.

Finally, we observe that there is an orthogonal (or strong)
factorisation system on $\cat{Tricat}_\mathrm s$ whose left
class comprises those strict homomorphisms which are bijective
on $0$-, $1$- and $2$-cells, and whose right class consists of
those strict homomorphisms which are locally locally fully
faithful; that is, those $F \colon \T \to \U$ for which the
following diagram of sets is a pullback:
\begin{equation*}
  \cd[@C+5em]{
    (V\T)_3 \ar[r]^{(VF)_3} \ar[d]_{(s, t)} &
    (V\U)_3 \ar[d]^{(s, t)} \\
    (V\T)_2 \times_{(V\T)_1} (V\T)_2 \ar[r]_{(VF)_2 \times_{(VF)_1} (VF)_2} & (V\U)_2 \times_{(V\U)_1} (V\U)_2\rlap{ .}
  }
\end{equation*}

We now give an explicit construction of the universal cofibrant
replacement $\epsilon_\T \colon Q \T \to \T$ of a tricategory
$\T$. We begin by defining $\T_1$ to be $LW\T$, the the free
tricategory on the underlying graph of $\T$, and \mbox{$e_1
\colon \T_1 \to \T$} to be the counit morphism. We now let
$\T_2$ be the tricategory obtained by adjoining the set of
$2$-cells
\begin{equation*}
    \set{[\alpha] \colon f \Rightarrow g}{\text{$f, g \colon X \to Y$ in $\T_1$ and $\alpha \colon e_1(f) \Rightarrow e_1(g)$ in $\T$}}
\end{equation*}
to $\T_1$, and define $e_2 \colon \T_2 \to \T$ to be the unique
strict homomorphism whose restriction to $\T_1$ is $e_1$, and
whose value at an adjoined $2$-cell $[\alpha] \colon f
\Rightarrow g$ is $\alpha \colon e_1(f) \Rightarrow e_1(g)$.
Finally, we obtain $Q\T$ and $\epsilon_\T$ by factorising $e_2$
as
\begin{equation}\label{fundamentalfac}
    e_2 = \T_2 \xrightarrow{\psi_\T} Q\T \xrightarrow{\epsilon_\T} \T
\end{equation}
where $\psi_\T$ is the identity on $0$-, $1$- and $2$-cells,
and $\epsilon_\T$ is locally locally fully faithful.

\begin{Prop}\label{counitchart}
The strict homomorphism $\epsilon_\T \colon Q\T \to \T$ is the
universal cofibrant replacement of $\T$.
\end{Prop}
\begin{proof}
We appeal to our recognition principle
Proposition~\ref{recognition}. First observe that a strict
homomorphism $F \colon \U \to \T$ may be equipped with a choice
of liftings with respect to the generating cofibrations only if
it is locally locally fully faithful; and that in this case, to
give such a choice is to give:
\begin{itemize}
\item For each $0$-cell $t \in \T$, a $0$-cell $k(t) \in
    \U$ with $Fk(t) = t$;
\item For each pair of $0$-cells $u, u'$ of $\U$ and each
    $1$-cell $f \colon Fu \to Fu'$ of $\T$, a $1$-cell
    $k(f,u,u') \colon u \to u'$ of $\U$ with $Fk(f,u,u') =
    f$;
\item For each parallel pair of $1$-cells $f, g \colon u
    \to u'$ of $\U$ and each $2$-cell $\alpha \colon Ff
    \Rightarrow Fg$ of $\T$, a $2$-cell $k(\alpha,f,g)
    \colon f \Rightarrow g$ of $\U$ with $Fk(\alpha,f,g) =
    \alpha$.
\end{itemize}
Observe now that $\epsilon_\T \colon Q \T \to \T$ is locally
locally fully faithful, and can be equipped with the following
choice of liftings:
\begin{itemize}
\item Since $Q \T$ has the same $0$-cells as $\T$, we may
    take $k(t) \defeq t$;
\item Since $Q \T$ has the same $1$-cells as $\T_1$, we may
    take $k(f, u, u') \defeq [f]$;
\item Since $Q \T$ has the same $2$-cells as $\T_2$, we may
    take $k(\alpha, f, g) \defeq [\alpha]$.
\end{itemize}
By Proposition~\ref{recognition}, if we can show that this data
determines an initial object of $\cat{AAF} / \T$, then we will
have shown $\epsilon_\T \colon Q \T \to \T$ to be the universal
cofibrant replacement of $\T$. So suppose $F \colon \V \to \T$
is another locally locally fully faithful strict homomorphism
equipped with a choice of liftings $k'$. From this we first
construct a strict homomorphism $H \colon \T_2 \to \V$; and to
do so, it suffices to specify where $H$ should sends each
$0$-cell $t$, each generating $1$-cell $[f] \colon t \to t'$,
and each generating $2$-cell $[\alpha] \colon f \Rightarrow g$.
So we take:
\begin{itemize}
   \item $H(t) = k'(t)$;
   \item $H([f] \colon t \to t') = k'(f, Ht, Ht')$;
   \item $H([\alpha] \colon f \to g) = k'(\alpha, Hf, Hg)$.
\end{itemize}
Now we observe that the outside of the following diagram commutes:
\begin{equation*}
    \cd{
        \T_2 \ar[r]^H \ar[d]_{\psi_\T} & \V \ar[d]^F \\
        Q\T \ar@{.>}[ur]\ar[r]_{\epsilon_\T} & \T
    }
\end{equation*}
and since $\psi_\T$ is bijective on $0$-, $1$- and $2$-cells,
and $F$ is locally locally fully faithful, it follows that
there is a unique strict homomorphism $K \colon Q \T \to \V$
(as indicated) rendering both induced triangles commutative.
It's now straightforward to prove that $K$ commutes with the
specified choices of lifting, and that moreover it is the
unique strict homomorphism that does so.
\end{proof}
Thus we have characterised the functor $Q$ and its counit
$\epsilon \colon Q \Rightarrow \id$; and it remains only to
describe the comultiplication $\Delta \colon Q \Rightarrow QQ$.
\begin{Prop}\label{comultt}
The strict homomorphism $\Delta_{\T} \colon Q{\T} \to QQ{\T}$
is uniquely determined by the following assignations:
\begin{itemize}
\item On $0$-cells, $\Delta_{\T}(t) = t$;
\item On $1$-cells, $\Delta_{\T}([f] \colon t \to t') =
    [[f]] \colon t \to t'$;
\item On $2$-cells, $\Delta_{\T}([\alpha] \colon f
    \Rightarrow g) = [[\alpha]] \colon \Delta_{\T}(f)
    \Rightarrow \Delta_{\T}(g)$;
\item On $3$-cells, $\Delta_{\T}(\Gamma \colon \alpha
    \Rrightarrow \beta) = \Gamma \colon \Delta_{\T}(\alpha)
    \Rrightarrow \Delta_{\T}(\beta)$.
\end{itemize}
\end{Prop}
\begin{proof}
Observe first that the above data determine a unique
homomorphism $K \colon Q\T \to QQ\T$ commuting with the maps
into $\T$. Therefore by Proposition~\ref{inducecomult} it
suffices to check that $K$ also commutes with the canonical
choices of liftings for these maps. For $\epsilon_\T$, these
liftings are given as in Proposition~\ref{counitchart}; whilst
for $\epsilon_\T . \epsilon_{Q\T}$, they are given as follows:
\begin{itemize}
\item On $0$-cells, $k(t) = t$;
\item On $1$-cells, $k(f, u, u') = [[f]]$;
\item On $2$-cells, $k(\alpha, f, g) = [[\alpha]]$.
\end{itemize}
These liftings are manifestly preserved by $K$, so that $K =
\Delta_X$ as required.
\end{proof}

\subsection{Trihomomorphisms}\label{tri3}
Recall that $\cat{Tricat}$, the category of tricategories and
trihomomorphisms, is defined to be the co-Kleisli category of
the comonad $\mathsf Q$. Our goal in the remainder of this
Section is to give an elementary description of this category
that does not require us to invoke the comonad $\mathsf Q$.

Now, morphisms $\T \to \U$ in $\cat{Tricat}$ are given by
strict homomorphisms $Q\T \to \U$, and so we wish to
characterise these latter maps in a manner that does not refer
to $Q$. First let us observe that precomposition with $\psi_T$
sends each such map to a strict homomorphism $\T_2 \to \U$, and
these latter have an easy characterisation: they are given by a
map $F \colon \T_1 \to \U$---which, since $\T_1 = LW\T$, is
equally well a map of underlying $1$-globular sets $W\T \to
W\U$---together with, for every pair of arrows $f, g \colon X
\to Y$ in $\T_1$ and $2$-cell $\alpha \colon e_1(f) \Rightarrow
e_1(g)$ in $\T$, a $2$-cell $F \alpha \colon Ff \Rightarrow Fg$
of $\U$. Thus, in order to characterise the trihomomorphisms
$\T \to U$, it will be enough to determine what extra data is
required in order to extend a strict homomorphism $\T_2 \to \U$
to one $Q\T \to \U$. However, the construction we have of $Q\T$
from $\T_2$, in terms of the
factorisation~\eqref{fundamentalfac}, is not suitable for this
purpose; and so we shall now give an alternative construction,
one that builds $Q\T$ from $\T_2$ through the adjunction of
$3$-cells and of $3$-cell equations.

Recall that to \emph{adjoin a $2$-cell} to a tricategory $\T$
is to take a pushout of the form~\eqref{adjoin2cell}. By
replacing the morphism $\iota_2 \colon \partial_2 \to 2_2$ in this
diagram with $\iota_3$ or $\iota_4$, we can say what it means
to \emph{adjoin a $3$-cell} or to \emph{adjoin a $3$-cell
equation} to $\T$: and hence what it means to adjoin an
\emph{invertible} $3$-cell to $\T$---namely, to adjoin
$3$-cells $\alpha \Rrightarrow \beta$ and $\beta \Rrightarrow
\alpha$ together with equations asserting these $3$-cells to be
mutually inverse. We shall now give a construction of $Q\T$
from $\T_2$ through the adjunction first of a number of
(invertible) $3$-cells, and then of a number of $3$-cell
equations.
\begin{Defn} The tricategory $\T_3$ is the result of
adjoining the following $3$-cells to $\T_2$:
\begin{itemize}
\item $3$-cells $[\Gamma] \colon [\alpha] \Rrightarrow
    [\beta] \colon f \Rightarrow g$ for $\Gamma \colon
    \alpha \Rightarrow \beta \colon e_2(f) \Rightarrow e_2
   (g)$ in $\T$;
\item Invertible 3-cells $V_{\alpha, \beta} \colon [\beta]
    \circ [\alpha] \Rrightarrow [\beta \circ \alpha] \colon
    f \Rightarrow h$ for $[\alpha] \colon f \Rightarrow g$
    and $[\beta] \colon g \Rightarrow h$ in $\T_2$;
\item Invertible 3-cells $H_{\alpha, \beta} \colon [\beta]
    \otimes [\alpha] \Rrightarrow [\beta \otimes \alpha]
    \colon h \otimes f \Rightarrow k \otimes g$ for
    $[\alpha] \colon f \Rightarrow g \colon x \to y$ and
    $[\beta] \colon h \Rightarrow k \colon y \to z$ in
    $\T_2$;
\item Invertible 3-cells $U_{f} \colon 1_f \Rrightarrow
    [1_{e_2(f)}] \colon f \Rightarrow f$ for $f \colon x
    \to y$ in $\T_2$;
\item Invertible 3-cells $L_f \colon l_f \Rrightarrow
    [l_{e_2(f)}] \colon I_y \otimes f \Rightarrow f$ for $f
    \colon x \to y$ in $\T_2$;
\item Invertible 3-cells $R_f \colon r_f \Rrightarrow
    [r_{e_2(f)}]  \colon f \otimes I_x \Rightarrow f$ for
    $f \colon x \to y$ in $\T_2$; and
\item Invertible 3-cells $A_{fgh} \colon a_{fgh}
    \Rrightarrow [a_{e_2(f),e_2( g),e_2 (h)}] \colon (h
    \otimes g) \otimes f \Rightarrow h \otimes (g \otimes
    f)$ for $f \colon x \to y$, $g \colon y \to z$ and $h
    \colon z \to w$ in $\T_2$.
\end{itemize}
\end{Defn}
The next step will be to adjoin a number of $3$-cell equations
to $\T_3$ to obtain a tricategory $\T_4$, which in
Proposition~\ref{t4iso} below we will be able to prove
isomorphic to $Q\T$. Before constructing $\T_4$, we give an
auxiliary definition which will make the task appreciably
simpler.
%
\begin{Defn}\label{definition}
For every $2$-cell $\gamma$ of $\T_2$, we define a $3$-cell
$\rho_\gamma \colon \gamma \Rrightarrow [e_2(\gamma)]$ of
$\T_3$ by structural induction over $\gamma$, exploiting the
fact that the $2$-cells under consideration are \emph{freely}
generated by those of the form $[\alpha] \colon f \Rightarrow
g$.
%
%
%
\begin{itemize}
\item If $\gamma = [\alpha] \colon f \Rightarrow g$ for
    some $\alpha \colon e_2(f) \Rightarrow e_2(g)$ in $\T$,
    we take $\rho_\gamma = \id_{[\alpha]}$;
\item If $\gamma = \beta \circ \alpha \colon f \Rightarrow
    h$ for some $\alpha \colon f \Rightarrow g$ and $\beta
    \colon g \Rightarrow h$, then we take $\rho_\gamma$ to
    be the composite
    \begin{equation*}
        \beta \circ \alpha \xrightarrow{\rho_\beta \circ \rho_\alpha} [e_2(\beta)] \circ [e_2(\alpha)] \xrightarrow{V_{e_2(\alpha), e_2(\beta)}} [e_2(\beta) \circ
        e_2(\alpha)] = [e_2(\beta \circ \alpha)]\ \text;
    \end{equation*}
\item If $\gamma = \beta \otimes \alpha \colon h \otimes f
    \Rightarrow k \otimes g$ for some $\alpha \colon f
    \Rightarrow g \colon x \to y$ and some $\beta \colon h
    \Rightarrow k \colon y \to z$, then we take
    $\rho_\gamma$ to be the composite
    \begin{equation*}
        \beta \otimes \alpha \xrightarrow{\rho_\beta \otimes \rho_\alpha} [e_2(\beta)] \otimes [e_2(\alpha)] \xrightarrow{H_{e_2(\alpha), e_2(\beta)}} [e_2(\beta) \otimes
        e_2(\alpha)] = [e_2(\beta \otimes \alpha)]\ \text;
    \end{equation*}
\item If $\gamma = 1_f \colon f \Rightarrow f$ for some $f
    \colon x \to y$, then we take $\rho_\gamma = U_f$;
\item If $\gamma = l_f$, $r_f$ or $a_{fgh}$, then we take
    $\rho_\gamma = L_f$, $R_f$ or $A_{fgh}$ respectively;
\item If $\gamma = l^\centerdot_f \colon f \Rightarrow I_y
    \otimes f$---where we recall
    from~\cite{Gurski2006algebraic} that such a $2$-cell
    participates in a specified adjoint equivalence
    $(\eta_f, \epsilon_f)$ with $l_f$---then we obtain
    $\rho_\gamma$ as follows. First we define a $3$-cell
    $\tilde\eta_f \colon 1_f \Rrightarrow [l_{e_2(f)}]
    \circ [l^\centerdot_{e_2(f)}]$ as the composite
\begin{equation}\label{etilde}
          1_f \xrightarrow{U_f}
          [1_{e_2(f)}] \xrightarrow{[\eta_{e_2(f)}]}
          [l_{e_2(f)} \circ l^\centerdot_{e_2(f)}] \xrightarrow{V^{-1}}
          [l_{e_2(f)}] \circ [l^\centerdot_{e_2(f)}]\ \text;
\end{equation}
and now we take $\rho_\gamma$ to be the pasting composite
\begin{equation}\label{paste}
    \cd[@+1em]{ \rtwocell[0.7]{dr}{\tilde \eta_f} &
      f \ar@{=}[r] \rtwocell[0.6]{dr}{L_f^{-1}} &
      f \ar[r]^{l^\centerdot_f} \rtwocell[0.3]{dr}{\epsilon_f} &
      I_y \otimes f \\
      f \ar[r]_{[l^\centerdot_{e_2(f)}]} \ar@{=}[ur] &
      I_y \otimes f \ar[u]_{[l_{e_2(f)}]} \ar@{=}[r] &
      I_y \otimes f \ar[u]_{l_f} \ar@{=}[ur] & {}
    }\ \text.
\end{equation}
The cases $\gamma = r^\centerdot_f$ and $\gamma =
a^\centerdot_{fgh}$ proceed analogously.
\end{itemize}
\end{Defn}


\begin{Defn}
\label{start} The tricategory $\T_4$ is obtained by adjoining
the following $3$-cell equalities to $\T_3$. First we force
compatibility with composition in every dimension.
\begin{itemize}
\item For each $[\Gamma] \colon [\alpha] \Rrightarrow
    [\beta]$ and $[\Delta] \colon [\beta] \Rrightarrow
    [\gamma]$, we require that
\begin{equation*}
  [\Delta] \circ [\Gamma] = [\Delta \circ \Gamma] \colon [\alpha]
  \Rrightarrow [\gamma]\ \text;
\end{equation*}
\item For each $[\alpha] \colon f \Rightarrow g$ we require
    that
\begin{equation*}
    \id_{[\alpha]} = [\id_\alpha] : [\alpha] \Rrightarrow [\alpha]\ \text;
\end{equation*}
\item For each $[\Gamma] \colon [\alpha] \Rrightarrow
    [\beta] \colon f \Rightarrow g$ and $[\Delta] \colon
    [\gamma] \Rrightarrow [\delta] \colon g \Rightarrow h$
    we require that the following diagram should commute:
\begin{equation*}
    \cd{
    [\gamma] \circ [\alpha] \ar[r]^{V} \ar[d]_{[\Delta] \circ [\Gamma]} &
    [\gamma \circ \alpha] \ar[d]^{[\Delta \circ \Gamma]} \\
    [\delta] \circ [\beta] \ar[r]_{V} &
    [\delta \circ \beta]\rlap{ ;}
    }
\end{equation*}
\item For each $[\Gamma] \colon [\alpha] \Rrightarrow
    [\beta] \colon x \to y$ and $[\Delta] \colon [\gamma]
    \Rrightarrow [\delta] \colon y \to z$ we require that
    the following diagram should commute:
\begin{equation*}
    \cd{
    [\gamma] \otimes [\alpha] \ar[r]^{H} \ar[d]_{[\Delta] \otimes [\Gamma]} &
    [\gamma \otimes \alpha] \ar[d]^{[\Delta \otimes \Gamma]} \\
    [\delta] \otimes [\beta] \ar[r]_{H} &
    [\delta \otimes \beta]\rlap{ .}
    }
\end{equation*}
\end{itemize}
The remaining equations we adjoin ensure compatibility between
the structural $3$-cells of $\T$ and those of the tricategory
we are defining. We begin by considering associativity and
unitality constraints in the hom-bicategories.
\begin{itemize}
\item For each $[\alpha] \colon f \rightarrow g$, we
    require that the following diagrams should commute:
\begin{equation*}
    \cd[@C+1em]{
      1_g \circ [\alpha] \ar[r]^-{\rho} \ar[d]_{\cong} &
      [1_{e_2(g)} \circ \alpha] \ar[d]^{[\mathord{\cong}]} \\
      [\alpha] \ar@{=}[r] &
      [\alpha]
    } \qquad \text{and} \qquad
    \cd[@C+1em]{
      [\alpha] \circ 1_f \ar[r]^-{\rho} \ar[d]_{\cong} &
      [\alpha \circ 1_{e_2(f)}] \ar[d]^{[\mathord{\cong}]} \\
      [\alpha] \ar@{=}[r] &
      [\alpha]\rlap{ ;}
    }
\end{equation*}
\item For each $[\alpha] \colon f \rightarrow g$, $[\beta]
    \colon g \rightarrow h$ and $[\gamma] \colon h
    \rightarrow k$, we require that the following diagram
    should commute:
\begin{equation*}
    \cd[@C+1em]{
      ([\gamma] \circ [\beta]) \circ [\alpha] \ar[r]^-{\rho} \ar[d]_{\cong} &
      [(\gamma \circ \beta) \circ \alpha] \ar[d]^{[\mathord{\cong}]} \\
      [\gamma] \circ ([\beta] \circ [\alpha]) \ar[r]_-{\rho} &
      [\gamma \circ (\beta \circ \alpha)]\rlap{ .}
    }
\end{equation*}
\end{itemize}
Next we require compatibility with the $3$-cells which mediate
middle-four interchange and its nullary analogue.
\begin{itemize}
\item For each suitable $[\alpha]$, $[\beta]$, $[\gamma]$
    and $[\delta]$ we require the following diagram to
    commute:
\begin{equation*}
    \cd[@C+1em]{
    ([\delta] \otimes [\beta]) \circ ([\gamma] \otimes [\alpha]) \ar[r]^-{\rho} \ar[d]_{\cong} &
    [(\delta \otimes \beta) \circ (\gamma \otimes \alpha)] \ar[d]^{[\mathord{\cong}]} \\
    ([\delta] \circ [\gamma]) \otimes ([\beta] \circ [\alpha]) \ar[r]_-{\rho} &
    [(\delta \circ \gamma) \otimes (\beta \circ \alpha)]\rlap{ ;}
    }
\end{equation*}
\item For each $f \colon x \to y$ and $g \colon y \to z$,
    we require the following diagram to commute:
\begin{equation*}
    \cd[@C+1em]{
    1_{g \otimes f} \ar[r]^-{\rho} \ar[d]_{\cong} &
    [1_{e_2(g) \otimes e_2(f)}] \ar[d]^{[\mathord{\cong}]} \\
    1_g \otimes 1_f \ar[r]_-{\rho} &
    [1_{e_2(g)} \otimes 1_{e_2(f)}]\rlap{ .}
    }
\end{equation*}
\end{itemize}
Next we ensure compatibility with the pseudonaturality
cells for the associativity and unitality constraints $a$,
$l$ and $r$.
\begin{itemize}
\item For all suitable $2$-cells $[\alpha] \colon f
    \Rightarrow m$, $[\beta] \colon g \Rightarrow n$
    and $[\gamma] \colon h \Rightarrow p$, we require
    that the following diagram should commute:
\begin{equation*}
    \cd{
      a_{m,n,p} \circ (([\gamma] \otimes [\beta]) \otimes [\alpha])
        \ar[r]^-{\rho} \ar[d]_{\cong} &
      [a_{e_2(m),e_2(n),e_2(p)} \circ ((\gamma \otimes \beta) \otimes \alpha)]
        \ar[d]^{[\mathord{\cong}]}\\
      ([\gamma] \otimes ([\beta] \otimes [\alpha])) \circ a_{f,g,h}
        \ar[r]_-{\rho} &
      [(\gamma \otimes (\beta \otimes \alpha)) \circ a_{e_2(f),e_2(g),e_2(h)}]\rlap{ ;}
    }
\end{equation*}
\item For each $[\alpha] \colon f \Rightarrow g \colon
    x \to y$, we require that the following diagrams
    should commute:
\begin{equation*}
    \cd{
    l_{g} \circ (I_y \otimes [\alpha]) \ar[r]^-{\rho} \ar[d]_{\cong} &
    [l_{e_2(g)} \circ (I_y \otimes \alpha)] \ar[d]^{[\mathord{\cong}]}\\
    [\alpha] \circ l_f \ar[r]_-{\rho} &
    [\alpha \circ l_{e_2(f)}]
    } \qquad \text{and} \qquad
    \cd{
    r_{g} \circ ([\alpha] \otimes I_x) \ar[r]^-{\rho} \ar[d]_{\cong} &
    [r_{e_2(g)} \circ (\alpha \otimes I_x)] \ar[d]^{[\mathord{\cong}]}\\
    [\alpha] \circ r_f \ar[r]_-{\rho} &
    [\alpha \circ r_{e_2(f)}]\rlap{ .}
    }
\end{equation*}
\end{itemize}
Finally, we ensure compatibility with the coherence $3$-cells
$\pi$ and $\mu$.
\begin{itemize}
\item For all composable $1$-cells $f$, $g$, $h$ and $k$,
    we require that the following diagram should commute:
\begin{equation*}
    \cd{
      (k \otimes a_{fgh}) \circ (a_{f, h \otimes g, k} \circ (a_{ghk} \otimes f))
        \ar[d]_\rho \ar[r]^-{\pi} &
      a_{g \otimes f, h, k} \circ a_{f, g, k \otimes h} \ar[d]^\rho \\
      [(\dot  k \otimes a_{\dot f, \dot g, \dot h}) \circ (a_{\dot  f, \dot h \otimes \dot  g, \dot k} \circ (a_{\dot g, \dot h, \dot k} \otimes \dot f))]
        \ar[r]_-{[\pi]} &
      [a_{\dot g \otimes \dot f, \dot h, \dot k} \circ a_{\dot f, \dot g, \dot k \otimes \dot h}]
   }
\end{equation*}
where we write $\dot f$ as an abbreviation for $e_2(f)$,
and so on; \vskip0.5\baselineskip
\item For all $1$-cells $f \colon x \to y$ and $g \colon y
    \to z$ we require that the following diagram should
    commute:
\begin{equation*}
    \cd{
      (g \otimes l_f) \circ (a_{f, I_y, g} \circ (r^\centerdot_g \otimes f))
        \ar[d]_\rho \ar[r]^-{\mu} &
      1_{g \otimes f} \ar[d]^\rho \\
      [(e_2(g) \otimes l_{e_2(f)}) \circ (a_{e_2(f), I_y, e_2(g)} \circ (r^\centerdot_{e_2(g)} \otimes e_2(f)))]
        \ar[r]_-{[\mu]} &
      [1_{e_2(g) \otimes e_2(f)}]\rlap{ .}
   }
\end{equation*}
\end{itemize}
\end{Defn}

\begin{Prop}\label{t4iso}
The tricategory $\T_4$ is isomorphic to $Q\T$ in
$\cat{Tricat}_\mathrm s$.
\end{Prop}
\begin{proof}
Let us write $\phi$ for the canonical map $\T_2 \to \T_4$. We
begin by factorising $e_2 \colon \T_2 \to \T$ as
\begin{equation}\label{psi2}
    \T_2 \xrightarrow{\phi} \T_4 \xrightarrow{e_4}  \T\ \text.
\end{equation}
To do so we must first specify where $e_4$ will take each of
the adjoined $3$-cells in $\T_4$; and then check that the
images under $e_4$ of the adjoined $3$-cell equations are
satisfied. We do this by sending each $3$-cell $[\Gamma] \colon
[\alpha] \Rrightarrow [\beta]$ to $\Gamma \colon \alpha
\Rrightarrow \beta$; and each of the other $3$-cells $U$, $V$,
$H$, $L$, $R$ and $A$ to the appropriate identity morphism.
It's easy to see that the requisite $3$-cell equations are then
satisfied, and so we obtain the desired
factorisation~\eqref{psi2}. We observe that $\phi$ is bijective
on $0$-, $1$-\ and $2$-cells, and so if we are able to show
$e_4$ to be locally locally fully faithful, then---by the
essential uniqueness of such factorisations---we can deduce the
existence of an isomorphism $\theta \colon Q\T \cong \T_4$ as
desired.

Thus, given $2$-cells $\gamma$ and $\delta$ of $\T_4$ we aim to
show that every $3$-cell $\Gamma \colon e_4(\gamma)
\Rrightarrow e_4(\delta)$ of $\T$ has the form $e_4(\tilde
\Gamma)$ for a unique $3$-cell $\tilde \Gamma \colon \gamma
\Rrightarrow \delta$ of $\T_4$. Now, by
Definition~\ref{definition}, we have invertible $3$-cells
$\rho_\gamma \colon \gamma \Rrightarrow [e_4(\gamma)]$ and
$\rho_\delta \colon \delta \Rrightarrow [e_4(\delta)]$, and by
structural induction can show that these maps are sent by $e_4$
to identity $3$-cells. Accordingly, the $3$-cell
\begin{equation*}
    \tilde \Gamma \defeq \gamma \xrightarrow{\rho_\gamma} [e_4(\gamma)] \xrightarrow{[\Gamma]}
      [e_4(\delta)] \xrightarrow{\rho_\delta^{-1}} \delta
\end{equation*}
of $\T_4$ satisfies $e_4(\tilde \Gamma) = \Gamma$; and it
remains only to show that it is unique with this property. We
shall do this by proving that, for every $3$-cell $\Delta
\colon \gamma \Rrightarrow \delta$ of $\T_4$, the following
square commutes:
\begin{equation}\label{commutativerho}
    \cd[@C+2em]{
      \gamma \ar[r]^{\Delta} \ar[d]_{\rho_\gamma} &
      \delta \ar[d]^{\rho_\delta}\\
      [e_4(\gamma)] \ar[r]_{[e_4(\Delta)]} &
      [e_4(\delta)]
    }\ \text;
\end{equation}
as then $e_4(\Delta) = \Gamma$ implies that $\Delta =
\rho_\delta^{-1} \circ [e_4(\Delta)] \circ \rho_\gamma =
\rho_\delta^{-1} \circ [\Gamma] \circ \rho_\gamma = \tilde
\Gamma$ as required. Since the $3$-cells of $\T_4$ are
generated---albeit not freely---by those of the form
$[\Gamma]$, $U$, $V$, $H$, $A$, $L$ and $R$, we may obtain
commutativity in~\eqref{commutativerho} by a structural
induction on the form of $\Delta$. The commutativity is
immediate when $\Delta$ is one of the generating $3$-cells
listed above; and has been explicitly adjoined in all cases
where $\Delta$ is a derived $3$-cell of $\T_4$, save for that
where $\Delta$ is a unit or counit map for one of the adjoint
equivalences $l^\centerdot_f \dashv l_f$, $r^\centerdot_f
\dashv r_f$ or $a^\centerdot_{fgh} \dashv a_{fgh}$. As a
representative sample of these cases, we show the square
\begin{equation*}
    \cd{
      1_{f} \ar[r]^-{\eta_f} \ar[d]_{\rho} & l_f \circ l^\centerdot_f \ar[d]^{\rho} \\
      [1_{e_2(f)}] \ar[r]_-{[\eta_f]} &
      [l_{e_2(f)} \circ l^\centerdot_{e_2(f)}]
    }
\end{equation*}
to be commutative. Writing $\tilde \eta_f$ for the $3$-cell
$1_f \Rrightarrow [l_{e_2 f}] \circ [l^\centerdot_{e_2 f}]$ of
equation~\eqref{etilde} and $L^\centerdot_f$ for the $3$-cell
$l^\centerdot_f \Rrightarrow [l^\centerdot_{e_2 f}]$
of~\eqref{paste}, this is equally well to show that
\begin{equation*}
    \cd[@C+1.5em]{
      1_{f} \ar[r]^-{\eta_f} \ar[d]_-{\tilde \eta_f} & l_f \circ l^\centerdot_f \ar[d]^{\id \circ L^\centerdot_f} \\
      [l_{e_2(f)}] \circ [l^\centerdot_{e_2(f)}] \ar[r]_-{L_f^{-1} \circ \id} &
      l_{f} \circ [l^\centerdot_{e_2(f)}]
    }
\end{equation*}
commutes; which follows by observing that $L^\centerdot_f$ is
the mate of $L_f^{-1}$ under the adjunctions $l^\centerdot_f
\dashv l_f$ and $[l^\centerdot_{e_2 f}] \dashv [l_{e_2 f}]$.
\end{proof}

We may now assemble all of the above calculations to give an
elementary description of the category $\cat{Tricat}$. In order
to give this without referencing the comonad $\mathsf Q$, we
will first need to introduce some notation. For objects $x, y$
of a tricategory $\T$, we define a \emph{formal composite of
$1$-cells} $f \colon x \dashrightarrow y$ by the following
clauses:
\begin{itemize}
\item If $x \in \T$ then $I_x \colon x \dashrightarrow x$;
\item If $f \colon x \to y$ in $\T$ then $[f] \colon x
    \dashrightarrow y$;
\item If $f \colon x \dashrightarrow y$ and $g \colon y
    \dashrightarrow z$ then $g \otimes f \colon x
    \dashrightarrow z$.
\end{itemize}
For each formal composite $f \colon x \dashrightarrow y$ we
recursively define its \emph{realisation} $\abs{f} \colon x \to
y$ by taking $\abs{\,[f]\,} = f$, $\abs{I_x} = I_x$ and $\abs{g
\otimes f} = \abs g \otimes \abs f$. Moreover, if given a
second tricategory $\U$ and a source-\ and target-preserving
assignation $F$ from the $0$- and $1$-cells of $\T$ to those of
$\U$, then we induce a mapping from formal composites $x
\dashrightarrow y$ to those $Fx \dashrightarrow Fy$ by another
recursion; we take $F[f] = [Ff]$, $FI_x = I_{Fx}$ and $F(g
\otimes f) = Fg \otimes Ff$. We may now give our elementary
restatement of the definition of $\cat{Tricat}$; that it is in
accordance with Definition~\ref{tricatdef1} is a direct
consequence of Propositions~\ref{comultt} and~\ref{t4iso}.

\begin{Defn}\label{tricatdef2}
The category $\cat{Tricat}$ has as its objects, the
tricategories of~\cite[Chapter~4]{Gurski2006algebraic}; whilst
its maps $F \colon \mathcal T \to \mathcal U$ are given by the
following basic data:
\begin{itemize}
\item For each $x \in \T$, an object $Fx \in \U$;
\item For each $f \colon x \to y$ of $\T$, a $1$-cell $Ff
    \colon Fx \to Fy$ of $\U$;
\item For each $f, g \colon x \dashrightarrow y$ and
    $\alpha \colon \abs f \Rightarrow \abs g$ in $\T$, a
    $2$-cell $F_{f,g}(\alpha) \colon \abs{Ff} \Rightarrow
    \abs{Fg}$ of $\U$;
\item For each $f, g \colon x \dashrightarrow y$, each
    $\alpha, \beta \colon \abs f \Rightarrow \abs g$ and
    each $\Gamma \colon \alpha \Rrightarrow \beta$ in $\T$,
    a $3$-cell $F_{f,g}(\Gamma) \colon F_{f,g}(\alpha)
    \Rrightarrow F_{f,g}(\beta)$ of $\U$;
\end{itemize}
and the following coherence data:
\begin{itemize}
\item For each $f, g, h \colon x \dashrightarrow y$,
    $\alpha \colon \abs f \Rightarrow \abs g$ and $\beta
    \colon \abs g \Rightarrow \abs h$ of $\T$, an
    invertible $3$-cell $V_{\alpha, \beta} \colon
    F_{g,h}(\beta) \circ F_{f,g}(\alpha) \Rrightarrow
    F_{f,h}(\beta \circ \alpha)$  of $\U$;
\item For each $f, g \colon x \dashrightarrow y$, $h, k
    \colon y \dashrightarrow z$, $\alpha \colon \abs f
    \Rightarrow \abs g$ and $\beta \colon \abs h
    \Rightarrow \abs k$ of $\T$, an invertible $3$-cell
    $H_{\alpha, \beta} \colon F_{h,k}(\beta) \otimes
    F_{f,g}(\alpha) \Rrightarrow F_{h \otimes f, k \otimes
    g}(\beta \otimes \alpha)$ of $\U$;
\item For each $f \colon x \dashrightarrow y$ of $\T$, an
    invertible $3$-cell $U_f \colon 1_{\abs{Ff}}
    \Rrightarrow F_{f,f}(1_{\abs f})$;
\item For each $f \colon x \dashrightarrow y$ in $\T$,
    invertible $3$-cells $L_f \colon l_{\abs{Ff}}
    \Rrightarrow F_{I_y \otimes f, f}(l_\abs{f})$ and $R_f
    \colon r_{\abs{Ff}} \Rrightarrow F_{f \otimes I_x,
    f}(r_\abs{f})$ of $\U$;
\item For each $f \colon x \dashrightarrow y$, $g \colon y
    \dashrightarrow z$ and $h \colon z \dashrightarrow w$
    in $\T$, an invertible $3$-cell $A_{fgh} \colon
    a_{\abs{Ff}, \abs{Fg}, \abs{Fh}} \Rrightarrow F_{(h
\otimes g) \otimes f, h \otimes (g \otimes f)}(a_{\abs{f},
\abs g, \abs h})$ of $\U$
\end{itemize}
subject to fourteen coherence axioms corresponding to the
fourteen kinds of $3$-cell equation adjoined in
Definition~\ref{start}. We give one of these axioms as a
representative sample. Suppose given $f, g \colon x
\dashrightarrow y$ and $\alpha \colon \abs f \Rightarrow
\abs g$ in $\T$. Then we require that
\begin{equation*}
    \cd[@C+1em]{
      1_{\abs{Fg}} \circ F_{f,g}(\alpha) \ar[r]^-{U_g \circ \id} \ar[d]_{\cong} &
      F_{g,g}(1_{\abs g}) \circ F_{f,g}(\alpha) \ar[r]^-{V_{\alpha, 1_{\abs g}}} &
      F_{f,g}(1_{\abs g} \circ \alpha) \ar[d]^{F_{f,g}(\cong)} \\
      F_{f,g}(\alpha) \ar@{=}[rr] & &
      F_{f,g}(\alpha)
    }
\end{equation*}
should commute in $\U$.

The identities and composition of $\cat{Tricat}$ are given as
follows. The identity homomorphism $\T \to \T$ has all of its
basic data given by identity assignations, and all of its
coherence data given by identity $3$-cells; whilst for
homomorphisms $F \colon \T \to \U$ and $G \colon \U \to \V$,
their composite $GF \colon \T \to \V$ has basic data given by
\begin{itemize}
\item $(GF)(x) = G(F(x))$;
\item $(GF)(f) = G(F(f))$;
\item $(GF)_{f,g}(\alpha) = G_{Ff, Fg}(F_{f,g}(\alpha))$;
\item $(GF)_{f,g}(\Gamma) = G_{Ff, Fg}(F_{f,g}(\Gamma))$;
\end{itemize}
and coherence data obtained according to a common pattern which
we illustrate with the case of $U_f$. Given $f \colon x
\dashrightarrow y$ in $\T$, we define the $3$-cell $U_f \colon
1_{\abs{GFf}} \Rrightarrow GF_{f, f}( 1_{\abs f})$ of $\V$ to
be the composite
\begin{equation*}
1_{\abs{GFf}} \xrightarrow{U_{Ff}} G_{Ff, Ff}(1_{\abs{Ff}}) \xrightarrow{G_{Ff,
Ff}(U_f)} G_{Ff, Ff}(F_{f, f}(1_{\abs f})) = GF_{f, f}(1_{\abs f})\ \text.
\end{equation*}
\end{Defn}

\section{Biased and unbiased trihomomorphisms}\label{s4}
As promised above, we now give a comparison between the notion
of trihomomorphism given in Definition~\ref{tricatdef2} and the
one already existing in the literature, a suitable reference
for which is~\cite[\S 3.3]{Gurski2006algebraic}. As observed
above, the two notions cannot be isomorphic, since our
trihomomorphisms admit a strictly associative composition,
whereas those of~\cite{Gurski2006algebraic} do not; at best,
they form a bicategory (see~\cite{Garner2008low-dimensional}
for the details). Closer inspection reveals that our
homomorphisms are the richer structure: they explicitly assign
to each two-dimensional pasting diagram of the domain
tricategory a corresponding pasting diagram in the codomain.
For the trihomomorphisms of~\cite{Gurski2006algebraic} no such
assignation is provided; and though one may be derived from the
trihomomorphism data---as we shall see in
Proposition~\ref{lifting} below---the derivation is
non-canonical, and so only determined up to an invertible
$3$-cell. A similar phenomenon occurs in comparing the
\emph{unbiased} bicategories of~\cite[Chapter
1]{Leinster2004Operads}---which incorporate specified
composites for all possible one-dimensional pasting
diagrams---with ordinary, or \emph{biased}, bicategories---for
which only nullary and binary composites are supplied. Again,
from the latter we can derive the former; but again, in a
non-canonical way that is determined only up to isomorphism. In
recognition of this similarity, we adopt
\cite{Leinster2004Operads}'s terminology here, referring to the
homomorphisms of Definition~\ref{tricatdef2} as~\emph{unbiased
homomorphisms}, and to those of~\cite[\S
3.3]{Gurski2006algebraic} as~\emph{biased homomorphisms}.

Our goal in the remainder of this section will be to give a
precise comparison between these two notions of homomorphism.
We will define a $2$-category of unbiased homomorphisms and a
bicategory of biased homomorphisms, and prove these to be
biequivalent. In each case, the $2$-cells we consider are not
the most general ones---those which between the biased
homomorphisms are called \emph{tritransformations}---since
these do not admit a strictly associative composition. Instead
we consider a restricted subclass of the tritransformations,
those whose $1$-\ and $2$-cell components are all identity
maps: these are the tricategorical \emph{icons}\footnote{In
fact, the icons we consider are
in~\cite{Garner2008low-dimensional} called \emph{ico-icons}:
with the unadorned name being reserved for a more general
concept which we will not have use of here.}
of~\cite{Garner2008low-dimensional}, themselves a
generalisation of the bicategorical icons
of~\cite{Lack20062-nerves}. Since the $1$-\ and $2$-dimensional
data for a tricategorical icon is trivial, it may be specified
purely in terms of a collection of $3$-cells satisfying axioms;
and it is this which allows us to equip them with a strictly
associative composition.

\begin{Defn}\label{icon}
Let $F, G \colon \T \to \U$ be unbiased homomorphisms. An
\emph{unbiased icon} $\xi \colon F \Rightarrow G$ may exist
only if $F$ and $G$ agree on $0$-\ and $1$-cells; and is then
given by specifying, for every $f, g \colon x \dashrightarrow
y$ and $\alpha \colon \abs f \Rightarrow \abs g$ in $\T$, a
$3$-cell $\xi_{f, g}(\alpha) \colon F_{f, g}(\alpha)
\Rrightarrow G_{f, g}(\alpha)$ of $\U$, subject to the
following axioms.
\begin{itemize}
\item For each $\Gamma \colon \alpha \Rrightarrow \beta
    \colon \abs f \Rightarrow \abs g$ of $\T$, the
    following diagram should commute in $\U$:
\begin{equation*}
   \cd{
     F_{f, g}(\alpha) \ar[r]^-{F_{f, g}(\Gamma)} \ar[d]_{\xi_{f, g}(\alpha)} &
     F_{f, g}(\beta) \ar[d]^{\xi_{f, g}(\beta)} \\
     G_{f, g}(\alpha) \ar[r]_-{F_{f, g}(\Gamma)} &
     G_{f, g}(\beta)\rlap{ ;}
   }
\end{equation*}
\item For each $f, g, h \colon x \dashrightarrow y$,
    $\alpha \colon \abs f \Rightarrow \abs g$ and $\beta
    \colon \abs g \Rightarrow \abs h$ of $\T$, the
    following diagram should commute in $\U$:
\begin{equation*}
   \cd{
     F_{g,h}(\beta) \circ F_{f,g}(\alpha) \ar[r]^-{V_{\alpha, \beta}} \ar[d]_{\xi_{g,h}(\beta) \circ \xi_{f,g}(\alpha)} &
     F_{f, h}(\beta \circ \alpha) \ar[d]^{\xi_{f,h}(\beta \circ \alpha)} \\
     G_{g,h}(\beta) \circ G_{f,g}(\alpha) \ar[r]_-{V_{\alpha, \beta}} &
     G_{f, h}(\beta \circ \alpha)
   }
\end{equation*}
\item For each $f, g \colon x \dashrightarrow y$, $h, k
    \colon y \dashrightarrow z$, $\alpha \colon \abs f
    \Rightarrow \abs g$ and $\beta \colon \abs h
    \Rightarrow \abs k$ of $\T$, the following diagram
    should commute in $\U$:
\begin{equation*}
   \cd{
     F_{h,k}(\beta) \otimes F_{f,g}(\alpha) \ar[r]^-{H_{\alpha, \beta}} \ar[d]_{\xi_{h,k}(\beta) \otimes \xi_{f,g}(\alpha)} &
     F_{h \otimes f, k \otimes g}(\beta \otimes \alpha) \ar[d]^{\xi_{h \otimes f, k \otimes g}(\beta \otimes \alpha)} \\
     G_{h,k}(\beta) \otimes G_{f,g}(\alpha) \ar[r]_-{H_{\alpha, \beta}} &
     G_{h \otimes f, k \otimes g}(\beta \otimes \alpha)\rlap{ ;}
   }
\end{equation*}
\item For each $f \colon x \dashrightarrow y$ in $\T$, the
    following diagrams should commute in $\U$:
\begin{equation*}
   \cd{
     1_{\abs{Ff}} \ar[r]^-{U_f} \ar@{=}[d] &
     F_{f, f}(1_{\abs f}) \ar[d]^{\xi} \\
     1_{\abs{Gf}} \ar[r]_-{U_g}  &
     G_{f, f}(1_{\abs f})
   } \ ,\
   \cd{
     l_{\abs{Ff}} \ar[r]^-{L_f} \ar@{=}[d] &
     F_{I_y \otimes f, f}(l_{\abs f}) \ar[d]^{\xi} \\
     l_{\abs{Gf}} \ar[r]_-{L_g}  &
     G_{I_y \otimes f, f}(l_{\abs f})
   } \ ,\
   \cd{
     r_{\abs{Ff}} \ar[r]^-{R_f} \ar@{=}[d] &
     F_{f \otimes I_x, f}(r_{\abs f}) \ar[d]^{\xi} \\
     r_{\abs{Gf}} \ar[r]_-{R_g}  &
     G_{f \otimes I_x, f}(r_{\abs f})\rlap{ ;}
   }
\end{equation*}
\item For each $f \colon x \dashrightarrow y$, $g \colon y
    \dashrightarrow z$ and $h \colon z \dashrightarrow w$
    in $\T$, the following diagram should commute in $\U$:
\begin{equation*}
   \cd{
     a_{\abs{Ff}, \abs{Fg}, \abs{Fh}}
     \ar[r]^-{A_{fgh}} \ar@{=}[d] &
     F_{f \otimes (g \otimes h), (f \otimes g) \otimes h}(a_{\abs{f}, \abs g, \abs h}) \ar[d]^{\xi} \\
     a_{\abs{Gf}, \abs{Gg}, \abs{Gh}}
     \ar[r]_-{A_{fgh}}  &
     G_{f \otimes (g \otimes h), (f \otimes g) \otimes h}(a_{\abs{f}, \abs g, \abs h})\rlap{ .}
   }
\end{equation*}
\end{itemize}
With the evident $2$-cell composition, tricategories,
unbiased homomorphisms and unbiased icons form a
$2$-category which we denote by
$\tcat{Tricat}_{\mathrm{ub}}$.
\end{Defn}

We now give the corresponding notion of icon between biased
homomorphisms. The definition is very similar to the one just
given, and we have deliberately stated it in a way which
facilitates easy comparison between the two. A more geometric
statement of the axioms is given
in~\cite[Definition~2]{Garner2008low-dimensional}.

\begin{Defn}\label{icon2}
Let $F, G \colon \T \to \U$ be biased homomorphisms. A
\emph{biased icon} $\xi \colon F \Rightarrow G$ may exist only
if $F$ and $G$ agree on $0$-\ and $1$-cells; and is then given
by specifying, for every $\alpha \colon f \Rightarrow g$ in
$\T$, a $3$-cell $\xi(\alpha) \colon F(\alpha) \Rrightarrow
G(\alpha)$ of $\U$; for every object $x \in \T$, an invertible
$3$-cell
\begin{equation*}
\cd{
    I_{Fx} \ar@2[r]^{\iota^F_x} \ar@{=}[d] \dthreecell{dr}{M_x} & FI_x \ar@{=}[d] \\
    I_{Gx} \ar@2[r]_{\iota^G_x} & GI_x
}
\end{equation*}
of $\U$; and for each pair of $1$-cells $f \colon x \to y$, $g
\colon y \to z$ of $\T$, an invertible $3$-cell
\begin{equation*}
\cd{
    Fg \otimes Ff \ar@2[r]^{\chi^F_{f, g}} \ar@{=}[d] \dthreecell{dr}{\Pi_{f, g}} & F(g \otimes f) \ar@{=}[d] \\
    Gg \otimes Gf \ar@2[r]_{\chi^G_{f, g}} & G(g \otimes f)
}
\end{equation*}
of $\U$, all subject to the following axioms.
\begin{itemize}
\item For each $\Gamma \colon \alpha \Rrightarrow \beta$ of
    $\T$, the following diagram should commute in $\U$:
\begin{equation*}
   \cd{
     F(\alpha) \ar[r]^-{F(\Gamma)} \ar[d]_{\xi(\alpha)} &
     F(\beta) \ar[d]^{\xi(\beta)} \\
     G(\alpha) \ar[r]_-{F(\Gamma)} &
     G(\beta)\rlap{ ;}
   }
\end{equation*}
\item For each $\alpha \colon f \Rightarrow g$ and $\beta
    \colon g \Rightarrow h \colon x \to y$ of $\T$, the
    following diagram should commute in $\U$:
\begin{equation*}
   \cd{
     F(\beta) \circ F(\alpha) \ar[r]^-{\cong} \ar[d]_{\xi(\beta) \circ \xi(\alpha)} &
     F(\beta \circ \alpha) \ar[d]^{\xi(\beta \circ \alpha)} \\
     G(\beta) \circ G(\alpha) \ar[r]_-{\cong} &
     G(\beta \circ \alpha)\rlap{ ;}
   }
\end{equation*}
\item For each $\alpha \colon f \Rightarrow g \colon x \to
    y$ and $\beta \colon h \Rightarrow k \colon y \to z$ of
    $\T$, the following diagram should commute in $\U$:
\begin{equation*}
   \cd{
     \chi^F_{g,k} \circ (F(\beta) \otimes F(\alpha)) \ar[r]^-{\cong} \ar[d]_{\Pi_{g,k} \circ (\xi(\beta) \otimes \xi(\alpha))} &
     F(\beta \otimes \alpha) \circ \chi^F_{f, h} \ar[d]^{\xi(\beta \otimes \alpha) \circ \Pi_{f,h}} \\
     \chi^G_{g,k} \circ (G(\beta) \otimes G(\alpha)) \ar[r]_-{\cong} &
     G(\beta \otimes \alpha) \circ \chi^G_{f, h}\rlap{ ;}
   }
\end{equation*}
\item For each $f \colon x \to y$ in $\T$, the following
    diagrams should commute in $\U$:
\begin{gather*}
   \cd{
     1_{Ff} \ar[r]^-{\cong} \ar@{=}[d] &
     F(1_{f}) \ar[d]^{\xi} \\
     1_{Gf} \ar[r]_-{\cong}  &
     G(1_f)
   }\quad \text, \quad
   \cd{
     l_{Ff} \ar@{<-}[r]^-{\gamma^F_f} \ar@{=}[d] &
     F(l_{f}) \circ (\chi^F_{f, I_y} \circ (\iota^F_y \otimes 1_{Ff})) \ar[d]^{\xi \circ (\Pi \circ (M \otimes \id))} \\
     l_{Gf} \ar@{<-}[r]_-{\gamma^G_f}  &
     G(l_{f}) \circ (\chi^G_{f, I_y} \circ (\iota^G_y \otimes 1_{Gf}))
   } \\
   \text{and} \quad \cd{
     r_{Ff} \ar@{<-}[r]^-{\delta^F_f} \ar@{=}[d] &
     F(r_f) \circ (\chi^F_{I_x, f} \circ (1_{Ff} \otimes \iota^F_x)) \ar[d]^{\xi \circ (\Pi \circ (\id \otimes M))} \\
     r_{Gf} \ar@{<-}[r]_-{\delta^G_f}  & G(r_f) \circ
     (\chi^G_{I_x, f} \circ (1_{Gf} \otimes
     \iota^G_x))\rlap{ ;}
   }
\end{gather*}
\item For each $f \colon x \to y$, $g \colon y\to z$ and $h
    \colon z \to w$ in $\T$, the following diagram should
    commute in $\U$:
\begin{equation*}
   \cd{
     (\chi^F_{g \otimes f, h} \circ (1_{Fh} \otimes \chi^F_{fg})) \circ a_{Ff, Fg, Fh}
     \ar@{<-}[r]^-{\omega^F_{fgh}} \ar[d]_{(\Pi \circ (\id \otimes \Pi)) \circ \id} &
     F(a_{fgh}) \circ (\chi^F_{f,g \otimes h} \circ (\chi^F_{gh} \otimes 1_{Ff})) \ar[d]^{\xi(a_{fgh}) \circ (\Pi \circ (\Pi \otimes \id))} \\
     (\chi^G_{g \otimes f, h} \circ (1_{Fh} \otimes \chi^G_{fg})) \circ a_{Gf, Gg, Gh}
     \ar@{<-}[r]_-{\omega^G_{fgh}} &
     G(a_{fgh}) \circ (\chi^G_{f,g \otimes h} \circ (\chi^G_{gh} \otimes 1_{Gf}))}\rlap{ .}
\end{equation*}
\end{itemize}
It follows from~\cite[Section~2]{Garner2008low-dimensional}
that tricategories, biased homomorphisms and biased icons
form a bicategory $\tcat{Tricat}_\mathrm b$.
\end{Defn}

We will now show $\tcat{Tricat}_{\mathrm{ub}}$ and
$\tcat{Tricat}_{\mathrm{b}}$ to be biequivalent. First we show
that every unbiased homomorphism $\T \to \U$ gives rise to a
biased homomorphism, and vice versa; then we show that these
assignations give rise to an equivalence of categories
$\tcat{Tricat}_{\mathrm{ub}}(\T, \U) \simeq
\tcat{Tricat}_{\mathrm{b}}(\T, \U)$; and finally, we show that
these equivalences provide the local data for an
identity-on-objects biequivalence $\tcat{Tricat}_{\mathrm{ub}}
\simeq \tcat{Tricat}_{\mathrm{b}}$.

\begin{Prop}\label{lifting2}
To each unbiased homomorphism $F \colon \T \to \U$ we may
assign a biased homomorphism $F' \colon \T \to \U$ with the
same action on $0$-\ and $1$-cells.
\end{Prop}
\begin{proof}
Suppose given an unbiased homomorphism $F \colon \T \to \U$. In
constructing the corresponding biased homomorphism $F'$, we
will give only the data and omit verification of the coherence
axioms, since these follow in a straightforward manner from the
axioms for $F$ and the tricategory axioms for $\U$. On $0$-\
and $1$-cells, $F'$ agrees with $F$; and on $2$-\ and $3$-cells
is given by:
\begin{equation*}
    F'(\alpha \colon f \Rightarrow g) = F_{[f],
    [g]}(\alpha) \quad \text{and} \quad
F'(\Gamma \colon \alpha \Rrightarrow \beta \colon f
    \Rightarrow g) = F_{[f], [g]}(\Gamma)\ \text.
\end{equation*}
The functoriality constraints for the homomorphisms of
bicategories $\T(x, y) \rightarrow \U(F'x, F'y)$ are given as
follows:
\begin{itemize}
\item For each $f \colon x \to y$ in $\U$, we take the
    constraint $3$-cell $1_{F'f} \cong F'(1_f)$ to be
    $U_{[f]} \colon 1_{Ff} \Rrightarrow F_{[f], [f]}(1_f)$;
\item For each $\alpha \colon f \Rightarrow g$ and $\beta
    \colon g \Rightarrow h$ in $\U$, we take the constraint
    $3$-cell $F'(\beta) \circ F'(\alpha) \cong F'(\beta
    \circ \alpha)$ to be $V_{\alpha, \beta} \colon F_{[g],
    [h]}(\beta) \circ F_{[f], [g]}(\alpha) \Rrightarrow
    F_{[f], [h]}(\beta \circ \alpha)$.
\end{itemize}
Next we provide the $2$-cell components of the pseudo-natural
transformations $\chi_{f,g}$ and $\iota_x$ and their adjoint
inverses $\chi^\centerdot_{f,g}$ and $\iota^\centerdot_x$. For
each $f \colon x \to y$ and $g \colon y \to z$ in $\T$ we take
\begin{align*}
    \chi_{f,g} &= F_{[g] \otimes [f], [g \otimes
    f]}(1_{g \otimes f}) \colon F'g \otimes F'f \Rightarrow F'(g \otimes
    f) \\
    \text{and} \quad \chi^\centerdot_{f,g} &= F_{[g \otimes f], [g] \otimes
    [f]}(1_{g \otimes f}) \colon F'(g \otimes f) \Rightarrow F'g \otimes
    F'f\ \text;
\end{align*}
whilst for each $x \in \T$ we take
\begin{align*}
    \iota_{x} &= F_{I_x, [I_x]}(1_{I_x}) \colon I_{F'x} \Rightarrow F'(I_x) \\
    \text{and} \quad \iota^\centerdot_{x} &= F_{[I_x], I_x}(1_{I_x}) \colon F'(I_x) \Rightarrow I_{F'x}\ \text.
\end{align*}
Given $2$-cells $\alpha \colon f \Rightarrow g \colon x \to y$
and $\beta \colon h \Rightarrow k \colon y \to z$ in $\T$, we
obtain the corresponding pseudonaturality $3$-cell for $\chi$
as the composite:
\begin{equation*}
\cd{
  \chi_{g,k} \circ (F'(\beta) \otimes F'(\alpha)) \ar[d]^= \\
  F_{[k] \otimes [g], [k \otimes
    g]}(1_{k \otimes g}) \circ (F_{[h], [k]}(\beta) \otimes F_{[f], [g]}(\alpha))
    \ar[d]^{\id \circ H} \\
  F_{[k] \otimes [g], [k \otimes
    g]}(1_{k \otimes g}) \circ F_{[h] \otimes [f], [k] \otimes [g]}(\beta \otimes \alpha)
    \ar[d]^{V} \\
  F_{[h] \otimes [f], [k \otimes
    g]}(1_{k \otimes g} \circ (\beta \otimes \alpha))
    \ar[d]^{F_{[h] \otimes [f], [k \otimes g]}(\mathord{\cong})} \\
  F_{[h] \otimes [f], [k \otimes
    g]}((\beta \otimes \alpha) \circ 1_{h \otimes f})
    \ar[d]^{V^{-1}} \\
  F_{[h \otimes f], [k \otimes g]}(\beta \otimes \alpha) \circ F_{[h] \otimes [f], [h \otimes f]}(1_{h \otimes f})
  \ \rlap{$ = F'(\beta \otimes \alpha) \circ \chi_{f,h}\ \text.$}
}
\end{equation*}
We next require unit and counit isomorphisms for the adjoint
equivalences $\chi^\centerdot \dashv \chi$ and
$\iota^\centerdot \dashv \iota$. So given $f \colon x \to y$
and $g \colon y \to z$ in $\T$, we obtain the isomorphism
$1_{F'(g \otimes f)} \Rrightarrow \chi_{f,g} \circ
\chi^\centerdot_{f,g}$ as the following composite:
\begin{equation*}
\cd{
  \llap{$1_{F'(g \otimes f)} = $} \ 1_{F(g \otimes f)}
    \ar[d]^{U} \\
  F_{[g \otimes f], [g \otimes f]}(1_{g\otimes f})
    \ar[d]^{F_{[g \otimes f], [g \otimes f]}(\mathord{\cong})} \\
  F_{[g \otimes f], [g \otimes f]}(1_{g \otimes f} \circ 1_{g \otimes f})
    \ar[d]^{V^{-1}} \\
  F_{[g] \otimes [f], [g \otimes f]}(1_{g \otimes f}) \circ F_{[g \otimes f], [g] \otimes [f]}(1_{g \otimes f})
  \ \rlap{$= \chi_{f,g} \circ \chi^\centerdot_{f,g}\ \text;$}
}
\end{equation*}
the other three cases are dealt with similarly. It remains only
to give the invertible modifications $\gamma$, $\delta$ and
$\omega$ witnessing the coherence of the functoriality
constraints $\chi$ and $\iota$. The same argument pertains in
each case, and so we give it only for $\gamma$. Here, for each
$f \colon x \to y$ in $\T$, we must give an invertible $3$-cell
\begin{equation*}
\cd[@!C@C-0.5em]{ &
  F'I_y \otimes Ff
    \ar[r]^{\chi_{f, I_y}}
    \dtwocell{dr}{\gamma_f} &
  F'(I_y \otimes f)
    \ar[dr]^{F'(l_f)} \\
  I_{F'y} \otimes F'f
    \ar[rrr]_{l_{F'f}}
    \ar[ur]^{\iota_y \otimes 1_{F'f}} & & &
  F'f\ \text;}
\end{equation*}
and we obtain this as the composite:
\begin{equation*}
\xymatrix{
  F_{[I_y \otimes f], [f]}(l_f) \circ \big(F_{[I_y] \otimes [f], [I_y \otimes f]}(1_{I_y \otimes f}) \circ (F_{I_y, [I_y]}(1_{I_y}) \otimes 1_{Ff})\big)
    \ar[d]^{\id \circ (\id \circ (\id \otimes U))} \\
  F_{[I_y \otimes f], [f]}(l_f) \circ \big(F_{[I_y] \otimes [f], [I_y \otimes f]}(1_{I_y \otimes f}) \circ (F_{I_y, [I_y]}(1_{I_y}) \otimes F_{[f], [f]}(1_f)\big)
    \ar[d]^{\id \circ (\id \otimes H)} \\
  F_{[I_y \otimes f], [f]}(l_f) \circ \big(F_{[I_y] \otimes [f], [I_y \otimes f]}(1_{I_y \otimes f}) \circ F_{I_y \otimes [f], [I_y] \otimes [f]}(1_{I_y} \otimes 1_f)\big)
    \ar[d]^{V . (\id \circ V)} \\
  F_{I_y \otimes [f], [f]}(l_f \circ (1_{I_y \otimes f} \circ (1_{I_y} \otimes 1_f)))
    \ar[d]^{F_{I_y \otimes [f], [f]}(\cong)} \\
  F_{I_y \otimes [f], [f]}(l_f)
    \ar[d]^{L_f^{-1}} \\
  l_{Ff} \ \text.
 }
\end{equation*}
\end{proof}
\begin{Prop}\label{lifting}
To each biased homomorphism $F \colon \T \to \U$ we may assign
an unbiased homomorphism $F' \colon \T \to \U$ with the same
action on $0$-\ and $1$-cells.
\end{Prop}
\begin{proof}
Let there be given a biased homomorphism $F \colon \T \to \U$.
We first define, for every $f \colon x \dashrightarrow y$ in
$\T$, an adjoint equivalence $2$-cell $\kappa_f \colon \abs{Ff}
\Rightarrow  F\abs f$ in $\U$. We do this by recursion on the
form of $f$.
\begin{itemize}
\item If $f = [g]$ for some $g \colon x \to y$, then we
    take $\kappa_f = \kappa^\centerdot_f = 1_{Fg} \colon Fg
    \Rightarrow Fg$.
\item If $f = I_x$ for some $x$, then we take $\kappa_f =
    \iota_x \colon I_{Fx} \Rightarrow FI_x$ and
    $\kappa^\centerdot_f = \iota^\centerdot_x$; and
\item If $f = h \otimes g$ for some $g \colon x
    \dashrightarrow z$ and $h \colon z \dashrightarrow y$,
    then we take $\kappa_f$ to be
\begin{equation*}
    \abs{F(h \otimes g)} = \abs {Fh} \otimes \abs{Fg} \xrightarrow{\kappa_h
    \otimes \kappa_g} F\abs {h} \otimes F\abs{g} \xrightarrow{\chi_{\abs g,
    \abs h}} F(\abs h \otimes \abs g) = F(\abs{h \otimes g})\ \text;
\end{equation*}
and give its adjoint inverse $\kappa^\centerdot_f$ dually.
\end{itemize}
We now define the unbiased homomorphism $F'$. To simplify
notation, we allow binary compositions to associate to the
right, and assert $0$-dimensional composition $\otimes$ to bind
more tightly than $1$-dimensional composition $\circ$. As
demanded by the Proposition, the basic data for $F'$ will agree
with that for $F$ on $0$-\ and $1$-cells; whilst on $2$-\ and
$3$-cells it is given by
\begin{equation*}
    F'_{f,g}(\alpha) = \kappa^\centerdot_g \circ F \alpha \circ \kappa_f
    \qquad \text{and} \qquad
F'_{f,g}(\Gamma) = \kappa^\centerdot_g \circ F \Gamma \circ \kappa_f\ \text.
\end{equation*}
The coherence data for $F'$ is given as follows. The invertible
$3$-cell $V_{\alpha, \beta} \colon F'_{g,h}(\beta) \circ
F'_{f,g}(\alpha) \Rrightarrow F'_{f,h}(\beta \circ \alpha)$ is
obtained as the chain of isomorphisms:
\begin{align*}
   & \mathrel{\phantom{\cong}} (\kappa^\centerdot_h \circ F\beta \circ \kappa_g) \circ (\kappa^\centerdot_g \circ F\alpha \circ \kappa_f)
   \\
   & \cong (\kappa^\centerdot_h \circ F\beta) \circ (\kappa_g \circ \kappa^\centerdot_g) \circ (F\alpha \circ \kappa_f)
   \\
   & \cong (\kappa^\centerdot_h \circ F\beta) \circ (F\alpha \circ \kappa_f)
   \\
   & \cong \kappa^\centerdot_h \circ F\beta \circ F\alpha \circ \kappa_f
   \\
   & \cong \kappa^\centerdot_h \circ F(\beta \circ \alpha) \circ \kappa_f
   \ \text;
\end{align*}
the invertible $3$-cell $H_{\alpha, \beta} \colon
F_{h,k}(\beta) \otimes F_{f,g}(\alpha) \Rrightarrow F_{h
\otimes f, k \otimes g}(\beta \otimes \alpha)$ by the chain of
isomorphisms:
\begin{align*}
   & \mathrel{\phantom{\cong}} (\kappa^\centerdot_k \circ F\beta \circ \kappa_h) \otimes (\kappa^\centerdot_g \circ F\alpha \circ
   \kappa_f) \\
   & \cong
   (\kappa^\centerdot_k \otimes \kappa^\centerdot_g) \circ (F\beta \otimes F\alpha) \circ (\kappa_h \otimes
   \kappa_f)\\
   & \cong
   (\kappa^\centerdot_k \otimes \kappa^\centerdot_g) \circ (\chi_{\abs g, \abs k}^\centerdot \circ F(\beta \otimes \alpha) \circ \chi_{\abs f, \abs h}) \circ (\kappa_h \otimes
   \kappa_f) \\
   & \cong
   (\kappa^\centerdot_k \otimes \kappa^\centerdot_g \circ \chi_{\abs g, \abs k}^\centerdot) \circ F(\beta \otimes \alpha) \circ (\chi_{\abs f, \abs h} \circ \kappa_h \otimes
   \kappa_f)\\
   & =
   \kappa^\centerdot_{k \otimes g} \circ F(\beta \otimes \alpha) \circ \kappa_{h \otimes
   f}
\end{align*}
(where from the second to the third line we apply
pseudonaturality of $\chi$); and the invertible $3$-cell $U_f
\colon 1_{\abs{Ff}} \Rrightarrow F_{f,f}(1_{\abs f})$ by the
chain of isomorphisms:
\begin{equation*}
   1_{\abs{Ff}}
    \xrightarrow{\cong} \kappa^\centerdot_f \circ \kappa_f
    \xrightarrow{\cong} \kappa^\centerdot_f \circ 1_{F\abs f} \circ \kappa_f
    \xrightarrow{\cong} \kappa^\centerdot_f \circ F(1_{\abs f}) \circ \kappa_f
   = F_{f,f}(1_{\abs f})\ \text.
\end{equation*}
It remains to give the invertible $3$-cells $L_f$, $R_f$ and
$A_{fgh}$. As these three cases are very similar, we give
details only for $L_f \colon l_{\abs{Ff}} \Rrightarrow F_{I_y
\otimes f, f}(l_\abs{f})$; which is obtained by the following
chain of isomorphisms:
\begin{align*}
   l_{\abs{Ff}}
   & \cong \kappa^\centerdot_f \circ l_{F\abs f} \circ 1 \otimes \kappa_f \\
   & \cong \kappa^\centerdot_f \circ (F(l_{\abs f}) \circ \chi_{\abs f, I_y} \circ \iota_y \otimes 1) \circ 1 \otimes
   \kappa_f \\
   & \cong \kappa^\centerdot_f \circ F(l_{\abs f}) \circ (\chi_{\abs f, I_y} \circ \iota_y \otimes \kappa_f) \\
   & = F_{I_y \otimes f, f}(l_\abs{f})\ \text,
\end{align*}
where for the first isomorphism we apply pseudonaturality of
$l$, and for the second we use the inverse of the coherence
$3$-cell
\begin{equation*}
\cd[@!C@C-0.5em]{ &
  FI_y \otimes F{\abs f}
    \ar[r]^{\chi_{{\abs f}, I_y}}
    \dtwocell{dr}{\gamma_{\abs f}} &
  F(I_y \otimes {\abs f})
    \ar[dr]^{F(l_{\abs f})} \\
  I_{Fy} \otimes F{\abs f}
    \ar[rrr]_{l_{F{\abs f}}}
    \ar[ur]^{\iota_y \otimes 1} & & &
  F{\abs f}\ \text.}
\end{equation*}
The fourteen coherence axioms for $F'$ now all follow from the
coherence theorem for biased
homomorphisms~\cite[Chapter~11]{Gurski2006algebraic}.
\end{proof}
\begin{Prop}\label{local}
The assignations of Propositions~\ref{lifting2}
and~\ref{lifting} induce an equivalence of categories
$\tcat{Tricat}_{\mathrm{ub}}(\T, \U) \simeq
\tcat{Tricat}_{\mathrm{b}}(\T, \U)$.
\end{Prop}
\begin{proof}
We begin by making the assignation of
Proposition~\ref{lifting2} into a functor. So suppose given an
unbiased icon $\xi \colon F \Rightarrow G$; we produce from it
a biased icon $\xi' \colon F' \Rightarrow G'$ as follows. We
take its basic data to be given by:
\begin{itemize}
\item $\xi'(\alpha \colon f \Rightarrow g) =
    \xi_{[f],[g]}(\alpha) \colon F'(\alpha) \Rrightarrow
    G'(\alpha)$;
\item $\Pi_{f,g} = \xi_{[g] \otimes [f],[g \otimes f]}(1_{g
    \otimes f}) \colon \chi^{F'}_{f,g} \Rrightarrow
    \chi^{G'}_{f,g}$; and
\item $M_x = \xi_{I_x,[I_x]}(1_{I_x}) \colon \iota^{F'}_x
    \Rrightarrow \iota^{G'}_x$.
\end{itemize}
The $3$-cells $\Pi_{f,g}$ and $M_x$ are invertible, with the
$3$-cell $\Pi_{f,g}^{-1}$ being given as the mate under
adjunction of the $3$-cell $\xi_{[g \otimes f],[g] \otimes
[f]}(1_{g \otimes f}) \colon \chi^{\centerdot\,F'}_{f,g}
\Rrightarrow \chi^{\centerdot\,G'}_{f,g}$, and $M_x^{-1}$ being
the mate under adjunction of $\xi_{[I_x],I_x}(1_{I_x}) \colon
\iota^{\centerdot\,F'}_x \Rrightarrow
\iota^{\centerdot\,G'}_x$; whilst the biased icon axioms for
$\xi'$ follow immediately from the unbiased icon axioms for
$\xi$. It is easy to see that the assignation $\xi \mapsto
\xi'$ is functorial, and so we obtain a functor $(\thg)' \colon
\tcat{Tricat}_{\mathrm{ub}}(\T, \U) \to
\tcat{Tricat}_{\mathrm{b}}(\T, \U)$.

We next make the assignation of Proposition~\ref{lifting} into
a functor. So given a biased icon $\xi \colon F \Rightarrow G$
we must produce an unbiased icon $\xi' \colon F' \Rightarrow
G'$. We first define, for every $f \colon x \dashrightarrow y$
in $\T$, invertible $3$-cells
\begin{equation*}
    \cd{
      \abs {Ff} \ar[r]^{\kappa^F_f} \ar@{=}[d] \dtwocell{dr}{\phi_f} &
      F\abs f \ar@{=}[d] \\
      \abs {Gf} \ar[r]_{\kappa^G_f} &
      G\abs f
    } \qquad \text{and} \qquad
    \cd{
      F\abs {f} \ar[r]^{\kappa^{\centerdot \, F}_f} \ar@{=}[d] \dtwocell{dr}{\phi^\centerdot_f} &
      \abs{Ff} \ar@{=}[d] \\
      G\abs {f} \ar[r]_{\kappa^{\centerdot \, G}_f} &
      \abs{Gf}
    } \ \text,
\end{equation*}
where $\kappa^F_f$, $\kappa^G_f$, $\kappa^{\centerdot \, F}_f$
and $\kappa^{\centerdot \, G}_f$ are defined as in the proof of
Proposition~\ref{lifting}. In fact, it suffices to give
$\phi_f$, since we may then obtain $\phi^\centerdot_f$ as the
mate under adjunction of $(\phi_f)^{-1}$. We define $\phi_f$ by
recursion on the form of $f$:
\begin{itemize}
\item If $f = [g]$ for some $g \colon x \to y$, then we
    take $\phi_f = \id \colon 1_{Fg} \Rrightarrow 1_{Gg}$;
\item If $f = I_x$ for some $x$, then we take $\phi_f = M_x
    \colon \iota^F_x \Rrightarrow \iota^G_x$; and
\item If $f = h \otimes g$ for some $g \colon x
    \dashrightarrow z$ and $h \colon z \dashrightarrow y$,
    then we take $\phi_f$ to be
\begin{equation*}
    \kappa^F_{h \otimes g} = \chi^F_{\abs g, \abs h} \circ (\kappa^F_h \otimes
    \kappa^F_g)
    \xrightarrow{\Pi_{\abs g, \abs h} \circ (\phi_h \otimes \phi_g)}
    \chi^G_{\abs g, \abs h} \circ (\kappa^G_h \otimes \kappa^G_g) = \kappa^G_{h
    \otimes g}\ \text.
\end{equation*}
\end{itemize}
Now for a biased icon $\xi \colon F \Rightarrow G$, the
corresponding unbiased icon $\xi' \colon F' \Rightarrow G'$ has
$3$-cell components $\xi'_{f, g}(\alpha)$ given by
\begin{equation*}
    F'_{f, g}(\alpha)
    = \kappa^{\centerdot\, F}_g \circ F(\alpha) \circ \kappa^F_f
    \xrightarrow{\phi^{\centerdot}_g \circ \xi(\alpha) \circ \phi_f}
    \kappa^{\centerdot\, G}_g \circ G(\alpha) \circ \kappa^G_f
    = G'_{f, g}(\alpha)\ \text.
\end{equation*}
The unbiased icon axioms for $\xi'$ follow by straightforward
diagram chasing. Moreover, it is easy to see that the
assignation $\xi \mapsto \xi'$ respects composition and so we
obtain a functor $(\thg)' \colon \tcat{Tricat}_{\mathrm{b}}(\T,
\U) \to \tcat{Tricat}_{\mathrm{ub}}(\T, \U)$.

It remains to show that the two functors just defined are
quasi-inverse to each other. Firstly, for each unbiased
homomorphism $F \colon \T \to \U$ we must provide an invertible
unbiased icon $\eta_F \colon F \Rightarrow F''$, naturally in
$F$. To this end we define, for each $f \colon x
\dashrightarrow y$ in $\T$, isomorphic $3$-cells
\begin{align*}
    \theta_f & \colon F_{f, [\,\abs f\,]}(1_{\abs f}) \Rrightarrow \kappa^{F'}_f
    \colon \abs{Ff} \Rightarrow F \abs f \\ \text{and} \quad
    \theta^\centerdot_f & \colon F_{[\,\abs f\,], f}(1_{\abs f}) \Rrightarrow \kappa^{\centerdot\,F'}_f
    \colon F\abs{f} \Rightarrow \abs{Ff}\text.
\end{align*}
We do this by recursion on the form of $f$. If $f = [g]$  then
we take
\begin{equation*}
\theta_f = \theta^\centerdot_f = U_{[g]}^{-1} \colon F_{[g],[g]}(1_g)
\Rrightarrow 1_{Fg}\ \text;
\end{equation*}
if $f = I_x$ for some $x$, then we may take both $\theta_f$ and
$\theta^\centerdot_f$ to be identity cells; and if $f = h
\otimes g$ for some $g \colon x \dashrightarrow z$ and $h
\colon z \dashrightarrow y$, then we take $\theta_f$ to be
given by the composite
\begin{align*}
    F_{h \otimes g, [\,\abs{h \otimes g}\,]}(1)
    & \cong F_{h \otimes g, [\,\abs{h} \otimes \abs {g}\,]}(1 \circ 1)\\
    & \cong F_{[\,\abs{h}\,] \otimes [\,\abs {g}\,], [\,\abs{h} \otimes \abs {g}\,]}(1) \circ F_{h \otimes g, [\,\abs h\,] \otimes [\,\abs
    g\,]}(1)\\
    & \cong F_{[\,\abs{h}\,] \otimes [\,\abs {g}\,], [\,\abs{h} \otimes \abs {g}\,]}(1) \circ F_{h, [\,\abs h\,]}(1) \otimes F_{g, [\,\abs g\,]}(1)\\
    & \cong \chi_{\abs f, \abs g}^{F'} \circ \kappa^{F'}_h \otimes \kappa^{F'}_g \\
    & = \kappa^{F'}_{h \otimes g}\ \text;
\end{align*}
and give $\theta^\centerdot_f$ dually. We now define the
unbiased icon $\eta_F \colon F \Rightarrow F''$ to have
components $(\eta_F)_{f,g}(\alpha)$ given by
\begin{equation*}
  \cd{
    F_{f,g}(\alpha)
      \ar[d]^{F_{f,g}(\cong)} \\
    F_{f,g}(1_{\abs g} \circ (\alpha \circ 1_{\abs f}))
      \ar[d]^{(1 \circ U^{-1}).U^{-1}} \\
    F_{[\,\abs g\, ],g}(1_{\abs g}) \circ (F_{[\,\abs f\,],[\,\abs g\,]}(\alpha) \circ F_{f,[\,\abs f\, ]}(1_{\abs f}))
      \ar[d]^{\theta^\centerdot_g \circ (\id \circ \theta_f)} \\
    \kappa^{\centerdot\,F'}_{g} \circ (F'(\alpha) \circ
    \kappa^{F'}_f)
      \ar[d]^{=} \\
    F''_{f,g}(\alpha) \ \text.}
\end{equation*}
With some effort we may check the icon axioms for $\eta_F$;
whilst the naturality of $\eta_F$ in $F$ is almost immediate.
To conclude the proof, we must provide for each biased
homomorphism $F \colon \T \to \U$ an invertible biased icon
$\epsilon_F \colon F'' \Rightarrow F$, naturally in $F$. Given
such an $F$, it is clear that $F''$ agrees with it on $0$-\ and
$1$-cells; whilst on $2$-cell data we have that:
\begin{align*}
    F''(\alpha \colon f \Rightarrow g) &= F'_{[f],[g]}(\alpha) = 1_{Fg} \circ (F\alpha \circ
    1_{Ff})\text; \\
    \chi^{F''}_{f,g} &= F'_{[g] \otimes [f], [g \otimes
    f]}(1_{g \otimes f}) = 1_{Fg} \otimes 1_{Ff} \circ \chi^F_{f,g} \circ F(1_{g \otimes f}) \circ
    1_{F(g \otimes f)}\text{;} \\
    \text{and }\ \  \iota^{F''}_{x} &= F'_{I_x, [I_x]}(1_{I_x}) = \iota^F_x \circ 1_{I_x} \circ
    1_{I_x}\text.
\end{align*}
Thus we may take each of $\epsilon_F(\alpha) \colon F''(\alpha)
\Rrightarrow F(\alpha)$, $\Pi_{f,g} \colon \chi^{F''}_{f,g}
\Rrightarrow \chi^{F}_{f,g}$ and $M_x \colon \iota^{F''}_x
\Rrightarrow \iota^F_x $ to be given by the appropriate
bicategorical coherence constraint. The icon axioms for
$\epsilon_F$ follow from coherence for biased trihomomorphisms;
whilst naturality of $\epsilon_F$ in $F$ is again almost
immediate.
\end{proof}

\begin{Thm}
The bicategories $\tcat{Tricat}_{\mathrm{ub}}$ and
$\tcat{Tricat}_{\mathrm{b}}$ are biequivalent.
\end{Thm}
\begin{proof}
We will show the functors $(\thg)' \colon
\tcat{Tricat}_{\mathrm{ub}}(\T, \U) \to
\tcat{Tricat}_{\mathrm{b}}(\T, \U)$ to provide the local
structure of an identity-on-objects homomorphism of
bicategories $\tcat{Tricat}_{\mathrm{ub}} \to
\tcat{Tricat}_{\mathrm b}$. The result then follows by
observing this homomorphism to be biessentially surjective on
objects (trivially) and locally an equivalence (by
Proposition~\ref{local}); and so a biequivalence. The only data
we lack for the homomorphism $\tcat{Tricat}_{\mathrm{ub}} \to
\tcat{Tricat}_{\mathrm b}$ are its functoriality constraint
$2$-cells. So we must provide for each tricategory $\T$, a
biased icon $e_\T \colon 1_\T \Rightarrow (1_\T)' \colon \T \to
\T$; and for each pair of unbiased homomorphisms $F \colon \T
\to \U$ and $G \colon \U \to \V$, a biased icon $m_{F,G} \colon
G' \circ F' \Rightarrow (G \circ F)' \colon \T \to \V$. For the
former, it is not hard to check that $(\thg)'$ in fact
preserves identities strictly, so that we may take $e_\T$ to be
an identity icon. For the latter, we observe that $G' \circ F'$
and $(G \circ F)'$ agree on $0$-\ and $1$-cells as required;
whilst on $2$-cells, their respective data is given as follows.
For $\alpha \colon f \Rightarrow g$ in $\T$, we have
\begin{align*}
    (G' \circ F')(\alpha) &= G'(F'(\alpha)) =
    G_{[Ff],[Fg]}(F_{[f],[g]}(\alpha)) \\
    \text{and} \ \ \ (G \circ F)'(\alpha) &= (G \circ F)_{[f],[g]}(\alpha) =
    G_{[Ff],[Fg]}(F_{[f],[g]}(\alpha))\text;
\end{align*}
so that we may take $m_{F,G}(\alpha)$ to be an identity
$3$-cell. Next, for $x \in \T$ we have
\begin{align*}
    \iota^{G' \circ F'}_{x} &= G'(\iota^{F'}_x) \circ \iota^{G'}_{F'x} =
    G_{[I_{Fx}], [FI_x]}(F_{I_x, [I_x]}(1_{I_x})) \circ G_{I_{Fx}, [I_{Fx}]}(1_{I_{Fx}})\\
    \text{and} \ \ \ \iota^{(G \circ F)'}_{x} &= (G \circ F)_{I_x, [I_x]}(1_{I_x}) = G_{I_{Fx}, [FI_x]}(1_{I_x})\text;
\end{align*}
so that we may take $M_x \colon \iota^{G' \circ F'}_{x}
\Rrightarrow \iota^{(G \circ F)'}_{x}$ to be the $3$-cell
\begin{equation*}
    \cd{
      G_{[I_{Fx}], [FI_x]}(F_{I_x, [I_x]}(1_{I_x})) \circ G_{I_{Fx}, [I_{Fx}]}(1_{I_{Fx}})
      \ar[d]^{V} \\
      G_{I_{Fx}, [FI_x]}(F_{I_x, [I_x]}(1_{I_x}) \circ 1_{I_{Fx}})
      \ar[d]^{G(\cong)} \\
      G_{I_{Fx}, [FI_x]}(F_{I_x, [I_x]}(1_{I_x}))\text.
    }
\end{equation*}
Finally, for $f \colon x \to y$ and $g \colon y \to z$ in $\T$,
we have that
\begin{align*}
    \chi^{G' \circ F'}_{f,g} &= G'(\chi^{F'}_{f,g}) \circ \chi^{G'}_{F'f,F'g} \\ & =
    G_{[Fg \otimes Ff], [F(g \otimes f)]}(F_{[g] \otimes [f],[g \otimes f]}(1_{g \otimes
    f})) \circ G_{[Fg] \otimes [Ff], [Fg \otimes Ff]}(1_{Fg \otimes Ff})\\
    \text{and} \ \ \ \chi^{(G \circ F)'}_{f,g} &= (G \circ  F)_{[g] \otimes [f],[g \otimes
    f]}(1_{g \otimes f}) \\ & = G_{[Fg] \otimes [Ff], [F(g \otimes f)]}(F_{[g] \otimes [f],[g \otimes
    f]}(1_{g \otimes f}))\text;
\end{align*}
so that we may take $\Pi_{f,g} \colon \chi^{G' \circ F'}_{f,g}
\Rrightarrow \chi^{(G \circ F)'}_{f,g}$ to be the $3$-cell
\begin{equation*}
    \cd{
      G_{[Fg \otimes Ff], [F(g \otimes f)]}(F_{[g] \otimes [f],[g \otimes f]}(1_{g \otimes
    f})) \circ G_{[Fg] \otimes [Ff], [Fg \otimes Ff]}(1_{Fg \otimes Ff})
      \ar[d]^{V} \\
      G_{[Fg] \otimes [Ff], [F(g \otimes f)]}(F_{[g] \otimes [f],[g \otimes f]}(1_{g \otimes
    f}) \circ 1_{Fg \otimes Ff})
      \ar[d]^{G(\cong)} \\
      G_{[Fg] \otimes [Ff], [F(g \otimes f)]}(F_{[g] \otimes [f],[g \otimes f]}(1_{g \otimes
    f}))\text.
  }
\end{equation*}
Finally, by straightforward diagram chasing we can verify in
succession: the icon axioms for $m_{F,G}$; naturality of
$m_{F,G}$ in $F$ and $G$; and the pentagon and triangle axioms
for $e_\T$ and $m_{F,G}$. This completes the definition of the
homomorphism $\tcat{Tricat}_{\mathrm{ub}} \to
\tcat{Tricat}_{\mathrm b}$ and hence the proof.
\end{proof}

\section{Homomorphisms of weak $\omega$-categories}\label{omcat}
We now turn to our second application of the techniques
described in Section~\ref{basicdef}, for which we shall develop
a notion of homomorphism between the weak $\omega$-categories
of Michael Batanin. These weak $\omega$-categories are defined
as algebras for suitable finitary monads on the category of
globular sets; and as such, the naturally-arising morphisms
between them are those which preserve all of the
$\omega$-categorical operations on the nose. Whilst
in~\cite[Definition 8.8]{Batanin1998Monoidal}, Batanin suggests
a way of weakening these maps to obtain a notion of
homomorphism, it is not made clear how the homomorphisms he
describes should be composed, or even that they may be composed
at all. The description that we shall now give of a
\emph{category} of homomorphisms between weak
$\omega$-categories is therefore a useful contribution towards
the goal of describing the totality of structure formed by
(algebraic) weak $\omega$-categories and the weak higher cells
between them.

We begin by briefly recalling Batanin's definition of weak
$\omega$-category: see~\cite{Batanin1998Monoidal}
or~\cite{Leinster2004Operads} for the details,
or~\cite{Batanin2008Algebras} for a more modern treatment. As
stated above, weak $\omega$-categories in this sense are
algebras for certain finitary monads on the category of
globular sets, where a \emph{globular set} is a presheaf over
the category $\cat G$ generated by the graph
\begin{equation*}
    \cd{
    0 \ar@<3pt>[r]^-{\sigma} \ar@<-3pt>[r]_-{\tau} & 1 \ar@<3pt>[r]^-{\sigma} \ar@<-3pt>[r]_-{\tau} & 2 \ar@<3pt>[r]^-{\sigma} \ar@<-3pt>[r]_-{\tau} & 3\ar@<3pt>[r]^-{\sigma} \ar@<-3pt>[r]_-{\tau} & \,\dots} \ \text,
\end{equation*}
subject to the equations $\sigma \sigma = \tau \sigma$ and
$\sigma \tau = \tau \tau$, and where the finitary monads in
question are the contractible globular operads
of~\cite{Batanin1998Monoidal}. A \emph{globular operad} is a
monad $P$ on $[\cat G^\op, \cat{Set}]$ equipped with a
cartesian monad morphism $\kappa \colon P \to T$, where $T$ is
the monad for strict $\omega$-categories, and where to call
$\kappa$ \emph{cartesian} is to assert that all of its
naturality squares are pullbacks. By Lemma~6.8 and
Proposition~6.11 of~\cite{Batanin2008Algebras}, any given monad
$P$ admits at most one such augmentation $\kappa$, so that for
a monad on $[\cat G^\op, \cat{Set}]$ to be a globular operad is
a property, not extra structure.\footnote{Note that this is by
contrast with the situation for plain operads, as noted
in~\cite{Leinster2006Are}.}

Since the identity monad on $[\cat G^\op, \cat{Set}]$ is a
globular operad, it is clear that not every globular operad
embodies a sensible theory of weak \mbox{$\omega$-categories}.
Those which do are characterised by~\mbox{\cite[Definition
8.1]{Batanin1998Monoidal}} in terms of a property
of~\emph{contractibility}. We will not recall the definition
here, because we will not need to: our development makes sense
for an arbitrary globular operad, and it will be convenient to
work at this level of generality. Thus, for the remainder of
this section, we let $P$ be a fixed globular operad.
\begin{Defn}
We write $\omega\text-\cat{Cat}_{\mathrm s}$ for the category
of $P$-algebras and $P$-algebra morphisms, refer to its objects
as \emph{weak $\omega$-categories}, and to its morphisms as
\emph{strict homomorphisms}.
\end{Defn}
The monad $T$ for strict $\omega$-categories is finitarily
monadic (see~\cite{Leinster2004Operads}), and this together
with the existence of a cartesian $\kappa \colon P \to T$
implies that $P$ is also finitary. Hence
$\omega\text-\cat{Cat}_{\mathrm s}$ is a locally finitely
presentable category, and so in order to apply the machinery of
Section~\ref{basicdef}, it remains only to distinguish in
$\omega\text-\cat{Cat}_{\mathrm s}$ a set of maps describing
the basic $n$-cells together with the inclusions of their
boundaries. In what follows we write
\begin{equation*} \cd[@C+1em]{
    \omega\text-\cat{Cat}_{\mathrm s} \ar@<5pt>@{<-}[r]^-{K} \ar@{}[r]|-{\bot} &
    [\cat{G}^\op, \cat{Set}] \ar@<5pt>@{<-}[l]^-{V}
}
\end{equation*}
for the free/forgetful adjunction induced by $P$.
\begin{Defn}\label{gencofibsom}
The \emph{generating cofibrations} $\{\iota_n \colon \partial_n
\to 2_n\}_{n \in \mathbb N}$ of $\omega\text-\cat{Cat}_\mathrm
s$ are the images under $K$ of the set of morphisms $\{f_n\}_{n
\in \mathbb N}$ of $[\cat G^\op, \cat{Set}]$ defined as follows
(where we write $y$ for the Yoneda embedding $\cat G \to [\cat
G^\op, \cat{Set}]$):
\begin{itemize}
\item $f_0$ is the unique map $0 \to y_0$;
\item $f_1$ is the map $[y_\sigma, y_\tau] \colon y_0 + y_0
    \to y_1$;
\item $f_n$ (for $n \geqslant 2$) is the map induced by the
    universal property of pushout in the following diagram:
\begin{equation*}
    \cd[@C+1em]{
        y_{n-2} + y_{n-2} \ar[r]^-{[y_\sigma,y_\tau]} \ar[d]_{[y_\sigma,y_\tau]} &
        y_{n-1} \ar[d] \ar[ddr]^{y_\tau}\\
        y_{n-1} \ar[r] \ar[drr]_{y_\sigma} &
        \star \pullbackcorner \ar@{.>}[dr]|{f_{n}} \\
        & & y_{n}
    }
\end{equation*}
\end{itemize}
\end{Defn}
\begin{Defn}\label{omcatdef}
We define $\mathsf Q \colon \omega\text-\cat{Cat}_\mathrm s \to
\omega\text-\cat{Cat}_\mathrm s$ to be the universal cofibrant
replacement comonad for the generating cofibrations of
Definition~\ref{gencofibsom}, and define the category
$\omega\text-\cat{Cat}_\mathrm s$ of weak $\omega$-categories
and $\omega$-homomorphisms to be the co-Kleisli category of
this comonad.
\end{Defn}

We shall now give an explicit description of the comonad
$\mathsf Q$ in terms of \emph{computads}. Computads were
introduced in~\cite{Street1976Limits} as a tool for presenting
free higher-dimensional categories. In the context of strict
$\omega$-categories they have been studied extensively under
the name of \emph{polygraph}:
see~\cite{Burroni1993Higher-dimensional,M'etayer2008Cofibrant}.
For the weak $\omega$-categories under consideration here, the
appropriate notion of computad is due to
Batanin~\cite{Batanin1998Computads}. In the definition, we make
use of the functors
\begin{align*}
    B_n & \defeq \omega\text-\cat{Cat}_{\mathrm s}(\partial_n, \thg) \colon \omega\text-\cat{Cat}_{\mathrm s} \to
    \cat{Set}\\
    \text{and} \quad E_n & \defeq \omega\text-\cat{Cat}_{\mathrm s}(2_n, \thg) \colon \omega\text-\cat{Cat}_{\mathrm s} \to \cat{Set}
\end{align*}
and the natural transformation $\rho_n \defeq
\omega\text-\cat{Cat}_{\mathrm s}(\iota_n, \thg) \colon E_n
\Rightarrow B_n$.

\begin{Defn}
For each integer $n \geqslant -1$, we define the category
$n\text-\cat{Cptd}$ of $n$-\emph{computads}, together with a
free/forgetful adjunction
\begin{equation*}
\cd[@C+1em]{
    \omega\text-\cat{Cat}_{\mathrm s} \ar@<5pt>@{<-}[r]^-{F_n} \ar@{}[r]|-{\bot} &
    n\text-\cat{Cptd} \ar@<5pt>@{<-}[l]^-{U_n}
}\ \text,
\end{equation*}
by induction on $n$. For the base case $n = -1$, we define
$(-1)$-$\cat{Cptd}$ to be the terminal category, $U_{-1}$ to be
the unique functor into it, and $F_{-1}$ to be the functor
picking out the initial weak $\omega$-category. For the
inductive step, given $n \geqslant 0$ we define an $n$-computad
to be given by an $(n-1)$-computad $C$, a set $X$, and a
function
\begin{equation*}
    x \colon X \to B_nF_{n-1}C\text.
\end{equation*}
A morphism of $n$-computads $(C, X, x) \to (C', X', x')$ is
given by a morphism $f \colon C \to C'$ of $(n-1)$-computads
and a map of sets $g \colon X \to X'$ making the diagram
\begin{equation*}
  \cd[@C+2em]{
    X \ar[r]^{g} \ar[d]_x & X' \ar[d]^{x'} \\
    B_nF_{n-1}C \ar[r]_-{B_nF_{n-1}f} &
     B_nF_{n-1}C'\text.
    }
\end{equation*}
commute. In other words, the category $n$-$\cat{Cptd}$ is just
the comma category $\cat{Set} \downarrow B_nF_{n-1}$. The
functor $U_n \colon \omega\text-\cat{Cat}_{\mathrm s} \to
n\text-\cat{Cptd}$ sends $\A$ to the triple $(U_{n-1}\A,
X_{(\A,n)}, x_{(\A,n)})$ where $X_{(\A,n)}$ and $x_{(\A,n)}$
are obtained from a pullback diagram
\begin{equation}\label{unpullback}
    \cd[@C+2em]{
      X_{(\A,n)} \ar[r]^{u_{(\A,n)}} \ar[d]_{x_{(\A,n)}} &
      E_n\A \ar[d]^{(\rho_n)_\A} \\
      B_n F_{n-1} U_{n-1} \A \ar[r]_-{B_n \epsilon_{n-1} \A} & B_n \A\rlap{ ;}
    }
\end{equation}
here $\epsilon_{n-1}$ denotes the counit of the adjunction
$F_{n-1} \dashv U_{n-1}$. To complete the definition, we must
exhibit a left adjoint $F_n$ for $U_n$. The value of this at an
$n$-computad $D = (C, X, x)$ is obtained by taking the
following pushout in $\omega$-$\cat{Cat}_\mathrm s$:
\begin{equation}\label{obtainalpha}
\cd{
    X \cdot \partial_n \ar[r]^{\overline x} \ar[d]_{X \cdot \iota_n} & F_{n-1} C \ar[d]^-{\psi_D} \\
    X \cdot 2_n \ar[r]_-{\phi_D} & F_n D\rlap{ ,}
    }
\end{equation}
where the map $\overline x$ is the transpose of $x \colon X \to
B_nF_{n-1} C$ under the adjunction $(\thg) \cdot \partial_n
\dashv B_n \colon \omega\text-\cat{Cat}_{\mathrm s} \to
\cat{Set}$.
The adjointness $F_n \dashv U_n$ follows by direct calculation.

For each natural number $n$, we have a functor $W_n \colon
(n+1)$-$\cat{Cptd} \to n\text-\cat{Cptd}$, sending $(C, X, x)$
to $C$; and the category $\omega$-$\cat{Cptd}$ of
\emph{$\omega$-computads} is defined to be the limit of the
diagram
\begin{equation*}
    \cdots \xrightarrow{W_1} 0\text-\cat{Cptd} \xrightarrow{W_0} (-1)\text-\cat{Cptd} \ \text.
\end{equation*}
For each $n \in \mathbb N$ we have $W_n U_n = U_{n-1}$, so that
the $U_n$'s form a cone over this diagram; and we write $U
\colon \omega\text-\cat{Cat}_\mathrm s \to
\omega\text-\cat{Cptd}$ for the induced comparison functor. It
now follows by a straightforward calculation that $U$ has a
left adjoint $F$, whose value at an object $(C_n)$ of
$\omega\text-\cat{Cptd}$ is given by the colimit of the diagram
\begin{equation*}
    F_{-1} C_{-1} \xrightarrow{\psi_{C_0}} F_0 C_0 \xrightarrow{\psi_{C_1}} F_1 C_1 \to \cdots
\end{equation*}
where the maps $\psi_{C_i}$ are given as
in~\eqref{obtainalpha}.
\end{Defn}
We now wish to show that the comonad $FU$ generated by the
adjunction \mbox{$F \dashv U \colon
\omega\text-\cat{Cat}_{\mathrm s} \to \omega\text-\cat{Cptd}$}
is isomorphic to the universal cofibrant replacement comonad
$\mathsf Q$. In order to do this, we will first need some
auxiliary definitions and results. Given a natural number $n$,
we define a morphism of globular sets $f \colon X \to Y$ to be
\emph{$n$-bijective} if $f_0, \dots, f_n$ are invertible, and
\emph{$n$-fully faithful} if the square
\begin{equation*}
  \cd[@C+3em]{
    X_{i+1} \ar[r]^{f_{i+1}} \ar[d]_{(s, t)} &
    Y_{i+1} \ar[d]^{(s, t)} \\
    X_i \times_{X_{i-1}} X_i \ar[r]_{f_i \times_{f_{i-1}} f_i} & Y_i \times_{Y_{i-1}} Y_i\rlap{ .}
  }
\end{equation*}
is a pullback for all $i \geqslant n$. We extend this notation
by declaring every morphism of $[\cat G^\op, \cat{Set}]$ to be
$(-1)$-bijective, and only the isomorphisms to be $(-1)$-fully
faithful.
\begin{Prop}
For each integer $n \geqslant -1$, there is an orthogonal
factorisation system on $\omega$-$\cat{Cat}_\mathrm s$ whose
left and right classes comprise those maps $f$ such that $Vf$
is $n$-bijective, respectively $n$-fully faithful.
\end{Prop}
\begin{proof}
The case $n = -1$ is trivial; so assume $n \geqslant 0$. It's
easy to show that the $n$-bijective and $n$-fully faithful maps
form an orthogonal factorisation system on $[\cat{G}^\op,
\cat{Set}]$; what we must show is that this lifts to
$\omega$-$\cat{Cat}_\mathrm s$. Since this latter is the
category of algebras for the monad $P$ on $[\cat{G}^\op,
\cat{Set}]$, it suffices for this to show that the functor $P$
preserves $n$-bijective morphisms. Indeed, if this is the case,
then we may factorise a $P$-algebra map $f \colon (X, x) \to
(Y, y)$ as follows. First we let
\begin{equation*}
    f = X \xrightarrow{g} Z \xrightarrow{h} Y
\end{equation*}
be the ($n$-bijective, $n$-fully faithful) factorisation of
$f$. Now consider the square
\begin{equation*}
   \cd{
    PX \ar[d]_{Pg} \ar[r]^-{g.x} & Z \ar[d]^h \\
    PZ \ar[r]_{y.Ph} & Y\rlap{ .}
   }
\end{equation*}
It is certainly commutative; and since $Pg$ is $n$-bijective
and $h$ is $n$-fully faithful, we induce a unique morphism $z
\colon PZ \to Z$ making both squares commute. It's now easy to
verify using the uniqueness of diagonal fillers, that this
makes $Z$ into a $P$-algebra, and $g$ and $h$ into $P$-algebra
maps. Thus we have verified the factorisation property; and the
lifting property may be verified similarly.

Thus to complete the proof it suffices to show that $P$
preserves $n$-bijective maps. But if $f \colon X \to Y$ is
$n$-bijective, then by direct examination, so is $Tf$ (where we
recall that $T$ is the monad for strict $\omega$-categories).
Now by virtue of the cartesian $\kappa \colon P \to T$, the map
$Pf$ is a pullback of the $n$-bijective $Tf$, and hence itself
$n$-bijective.
\end{proof}

\begin{Prop}\label{propcomp}
For any natural number $n$ and $n$-computad $D = (C, X, x)$,
the map $\psi_D \colon F_{n-1}C \to F_nD$ of
equation~\eqref{obtainalpha} is $(n-1)$-bijective.
\end{Prop}
\begin{proof}
The case $n = 0$ is trivial; so suppose $n \geqslant 1$. In
this case, the map $\psi_D$ is a pushout of a coproduct of
copies of $\iota_n \colon
\partial_n \to 2_n$, and so---by standard properties of orthogonal factorisation systems---will be
$(n-1)$-bijective so long as $\iota_{n}$ is. But we defined
$\iota_{n}$ to be the image under the free functor $K$ of the
map $f_{n} \in [\cat G^\op, \cat{Set}]$, and so the result
follows by observing that $K$ preserves $(n-1)$-bijectives
(because $P$ does), and that $f_{n}$ is $(n-1)$-bijective by
direct examination.
\end{proof}
With these preliminaries in place, we may now prove our main
result.
\begin{Prop}
The comonad $\mathsf Q$ is isomorphic to the comonad generated
by the adjunction $F \dashv U \colon
\omega\text-\cat{Cat}_{\mathrm s} \to \omega\text-\cat{Cptd}$.
\end{Prop}
\begin{proof}
Let there be given a weak $\omega$-category $\A$. We will use
Proposition~\ref{recognition} to show that the counit morphism
$\epsilon_\A \colon FU\A \to \A$ provides a universal cofibrant
replacement of $\A$. Thus we must equip $\epsilon_\A$ with a
choice of liftings against the generating cofibrations which
makes it into an initial object of $\cat{AAF} / \A$, the
category of algebraic acyclic fibrations into $\A$.

We first observe that to equip a strict homomorphism $f \colon
\X \to \A$ with a choice of liftings against the generating
cofibrations is to give, for each $n \in \mathbb N$, a section
of the function $((\rho_n)_\X, E_nf) \colon E_n\X \to B_n\X
\times_{B_n\A} E_n\A$. Thus to equip $\epsilon_\A \colon FU\A
\to \A$ with a choice of liftings is to give functions
\begin{equation}\label{kn}
    k_n \colon B_n FU\A \times_{B_n\A} E_n \A \to E_n FU\A
\end{equation}
for each $n \in \mathbb N$ such that $ ((\rho_n)_{FU\A}, E_n
\epsilon_\A) \circ k_n$ is the identity. Now, $FU\A$ is
obtained as the following colimit:
\begin{equation*}
    \cd[@C+1em]{
      F_{-1} U_{-1} \A \ar[r]^{\psi_{U_0 \A}} \ar[dr]_{\alpha_{-1}} & F_0 U_0 \A \ar[r]^{\psi_{U_1\A}} \ar[d]_{\alpha_0} & F_1 U_1 \A \ar[r] \ar[dl]^{\alpha_1} & \cdots \\
      & F U \A
    }
\end{equation*}
and $\epsilon_\A \colon FU\A \to \A$ as the unique map
satisfying $\epsilon_\A . \alpha_n = (\epsilon_n)_\A$ for all
$n \geqslant -1$. Given $n \in \mathbb N$, we have by
Proposition~\ref{propcomp} that $\psi_{{U_m}\A}$ is
$(n-1)$-bijective for each $m \geqslant n$, from which it
follows by standard properties of orthogonal factorisation
systems that $\alpha_{n-1}$ is also $(n-1)$-bijective.
Moreover, the functor $B_n \colon \omega\text-\cat{Cat}_\mathrm
s \to \cat{Set}$ sends $(n-1)$-bijectives to isomorphisms, so
that $B_n \alpha_{n-1} \colon B_n F_{n-1} U_{n-1} \A \to B_n
FU\A$ is an isomorphism: and so composing the pullback
square~\eqref{unpullback} with this map yields a pullback
square
%
%
\begin{equation*}
    \cd[@C+2em]{
      X_{(\A,n)} \ar[r]^{u_{(\A,n)}} \ar[d]_{B_n \alpha_{n-1} \circ x_{(\A,n)}} &
      E_n\A \ar[d]^{(\rho_n)_\A} \\
      B_n FU\A \ar[r]_-{B_n \epsilon_\A} & B_n \A\rlap{ .}
    }
\end{equation*}
So to give $k_n$ as in~\eqref{kn} is equally well to give $k'_n
\colon X_{(\A, n)} \to E_n FU\A$ such that
\begin{equation}\label{equalities}
    (\rho_n)_{FU\A} \circ k'_n = B_n \alpha_{n-1} \circ x_{(\A, n)} \qquad \text{and} \qquad E_n \epsilon_\A \circ k'_n = u_{(\A, n)}\rlap{ ;}
\end{equation}
and we may obtain such a $k'_n$ as the transpose of the
composite
\begin{equation*}
    X_{(\A, n)} \cdot 2_n \xrightarrow{\phi_{U_n \A}} F_n U_n \A \xrightarrow{\alpha_n} FU\A\rlap{ ,}
\end{equation*}
under the adjunction $(\thg) \cdot 2_n \dashv E_n \colon
\omega\text-\cat{Cat}_{\mathrm s} \to \cat{Set}$. A
straightforward calculation now verifies the equalities
in~\eqref{equalities}.

Thus we have equipped $\epsilon_\A$ with a choice of liftings
$(k_n)$ against the generating cofibrations; it remains to show
that this makes $(\epsilon_\A, k_n)$ into an initial object of
$\cat{AAF} / \A$. So suppose that $f \colon \X \to \A$ is a
strict homomorphism equipped with a choice of liftings $j_n
\colon B_n\X \times_{B_n\A} E_n\A \to E_n\X$. We shall define a
morphism $\beta \colon FU\A \to \X$ satisfying $f.\beta =
\epsilon$. To do so is equally well to give a cocone
\begin{equation}\label{cocone}
    \cd[@C+1em]{
      F_{-1} U_{-1} \A \ar[r]^{\psi_{U_0 \A}} \ar[dr]_{\beta_{-1}} & F_0 U_0 \A \ar[r]^{\psi_{U_1\A}} \ar[d]_{\beta_0} & F_1 U_1 \A \ar[r] \ar[dl]^{\beta_1} & \cdots \\
      & \X
    }
\end{equation}
satisfying $f . \beta_n = \epsilon_n$ for all $n \geqslant -1$.
We do so by recursion on $n$. For the base case, we take
$\beta_{-1}$ to be the unique map from the initial object
$F_{-1}U_{-1} \A$. For the inductive step, let $n \geqslant 0$
and suppose that we have already defined $\beta_{n-1}$
satisfying $f . \beta_{n-1} = \alpha_{n-1}$. By virtue of the
pushout diagram~\eqref{obtainalpha} and the requirement
that~\eqref{cocone} should be a cocone, to give $\beta_n \colon
F_n U_n \A \to \X$ is equally well to give a morphism $b_n
\colon X_{(\A, n)} \cdot 2_n \to \X$ making the square
\begin{equation*}
\cd{
    X_{(\A, n)} \cdot \partial_n \ar[r]^{\overline{x_{(\A, n)}}} \ar[d]_{X \cdot \iota_n} & F_{n-1} U_{n-1} \A \ar[d]^-{\beta_{n-1}} \\
    X_{(\A, n)} \cdot 2_n \ar[r]_-{b_n} & \X
   }
\end{equation*}
commute; which, taking transposes under adjunction, is equally
well to give a morphism $b'_n \colon X_{(\A, n)} \to E_n\X$
making
\begin{equation}\label{newcommute}
\cd{
    X_{(\A, n)} \ar[r]^-{x_{(\A, n)}} \ar[d]_{b'_n} & B_n F_{n-1} U_{n-1} \A \ar[d]^-{B_n\beta_{n-1}} \\
    E_n\X \ar[r]_-{(\rho_n)_\X} & B_n\X
   }
\end{equation}
commute. To do so, we consider the following diagram:
\begin{equation}\label{xwz}
\cd{
    X_{(\A, n)} \ar[r]^-{u_{(\A, n)}} \ar[d]_{B_n \beta_{n-1} \circ x_{(\A,n)}} & E_n \A \ar[d]^-{(\rho_n)_\A} \\
    B_n\X \ar[r]_-{B_n f} & B_n\A\rlap{ .}
   }
\end{equation}
It commutes by~\eqref{unpullback} and the condition
$f.\beta_{n-1} = \epsilon_{n-1}$, and so we induce a map
$X_{(A, n)} \to B_n\X \times_{B_n\A} E_n\A$ by universal
property of pullback. We now define $b'_n$ to be the composite
of this with $j_n \colon B_n\X \times_{B_n\A} E_n\A \to E_n\X$.
Commutativity in~\eqref{xwz}, together with the fact that $j_n$
is a section now imply both that~\eqref{newcommute} is
commutative and that $f . \beta_n = \alpha_n$ as required. This
completes the construction of $\beta \colon FU\A \to \X$; and
further calculation now shows that this map preserves the
choices of liftings for $\epsilon_\A$ and for $f$, and
moreover, that it is the unique morphism $FU\A \to \X$ over
$\A$ with this property.

Therefore, by Proposition~\ref{recognition}, we have shown that
the functor and counit part of the universal cofibrant
replacement comonad coincide (up to isomorphism) with the
functor and counit part of the comonad induced by the
adjunction \mbox{$F \dashv U \colon
\omega\text-\cat{Cat}_{\mathrm s} \to \omega\text-\cat{Cptd}$}.
To show that the same is true for the comultiplication is a
long but straightforward calculation using
Proposition~\ref{comultt} which we omit.
\end{proof}

We end the paper with some brief remarks on higher cells. We
have a functor $D \colon \cat{G} \to
\cat{\omega\text-\cat{Cat}}$ obtained as the composite
\begin{equation*}
    \cat{G} \xrightarrow{y_{(\thg)}} \hat{\cat{G}} \xrightarrow{\text{free}} T\text-\cat{Alg} \xrightarrow{\kappa^\ast} P\text-\cat{Alg} = \cat{\omega\text-\cat{Cat}}_\mathrm s\ \hookrightarrow \omega\text-\cat{Cat}\text.
\end{equation*}
Observe that $\omega\text-\cat{Cat}$ has products---because
$\omega\text-\cat{Cat}_\mathrm s$ has them and the inclusion
map preserves them, being a right adjoint---so that, as in
\cite[Definition 8.9]{Batanin1998Monoidal}, we may define an
\emph{$m$-cell} from $\A$ to $\B$ to be a homomorphism $\A
\times Dm \to \B$. Whilst it is unclear how one should compose
such $m$-cells in general, there is one form of composition we
do have: namely, that along a $0$-cell boundary.
\begin{Prop}
There is a category $\omega\text-\cat{Cat}_m$ whose objects are
weak $\omega$-categories and whose morphisms $\A \to \B$ are
$m$-cells from $\A$ to $\B$.
\end{Prop}
\begin{proof}
Take the co-Kleisli category of the comonad $(\thg) \times Dm$
on $\omega\text-\cat{Cat}$.
\end{proof}
\begin{Cor}
$\omega$-$\cat{Cat}$ is enriched over the cartesian monoidal
category of globular sets.
\end{Cor}
\begin{proof}
The hom object from $\A$ to $\B$ is the globular set $[\A, \B]$
with
\begin{equation*}
    [\A, \B]_m \defeq \cat{\omega\text-\cat{Cat}}(\A \times Dm,\, \B)\ \text;
\end{equation*}
whilst composition and identities at dimension $m$ are given as
in $\omega$-$\cat{Cat}_m$.
\end{proof}


\end{document}